\documentclass[reqno]{amsart}

\usepackage{amsmath}
\usepackage{amsthm}
\usepackage{amscd}
\usepackage{bbm}
\usepackage{enumerate}
\usepackage{amssymb}
\usepackage{palatino}
\usepackage{amscd}
\usepackage{verbatim}
\usepackage{mathrsfs}
\usepackage[margin=1.5in]{geometry}

\numberwithin{equation}{section}



\title{The average number of integral points on elliptic curves is bounded}
\author{Levent Alpoge}\email{levent.alpoge@gmail.com}
\address{Churchill College, University of Cambridge, Cambridge CB3 0DS.}

\begin{document}

\begin{abstract}
We prove that, when elliptic curves $E/\mathbb{Q}$ are ordered by height, the average number of integral points $\#|E(\mathbb{Z})|$ is bounded, and in fact is less than $66$ (and at most $\frac{8}{9}$ on the minimalist conjecture). By ``$E(\mathbb{Z})$'' we mean the integral points on the corresponding quasiminimal Weierstrass model $E_{A,B}: y^2 = x^3 + Ax + B$ with which one computes the na\"{\i}ve height. The methods combine ideas from work of Silverman, Helfgott, and Helfgott-Venkatesh with work of Bhargava-Shankar and a careful analysis of local heights for ``most'' elliptic curves. The same methods work to bound integral points on average over the families $y^2 = x^3 + B$, $y^2 = x^3 + Ax$, and $y^2 = x^3 - D^2 x$.
\end{abstract}

\maketitle

\newtheoremstyle{dotless}{}{}{\itshape}{}{\bfseries}{}{ }{}

\newtheorem{thm}{Theorem}
\newtheorem{lem}[thm]{Lemma}
\newtheorem{remark}[thm]{Remark}
\newtheorem{cor}[thm]{Corollary}
\newtheorem{defn}[thm]{Definition}
\newtheorem{prop}[thm]{Proposition}
\newtheorem{conj}[thm]{Conjecture}
\newtheorem{claim}[thm]{Claim}
\newtheorem{exer}[thm]{Exercise}
\newtheorem{fact}[thm]{Fact}

\theoremstyle{dotless}

\newtheorem{thmnodot}[thm]{Theorem}
\newtheorem{lemnodot}[thm]{Lemma}
\newtheorem{cornodot}[thm]{Corollary}

\newcommand{\image}{\mathop{\text{image}}}
\newcommand{\End}{\mathop{\text{End}}}
\newcommand{\Hom}{\mathop{\text{Hom}}}
\newcommand{\Sum}{\displaystyle\sum\limits}
\newcommand{\Prod}{\displaystyle\prod\limits}
\newcommand{\Tr}{\mathop{\mathrm{Tr}}}
\renewcommand{\Re}{\operatorname{\mathfrak{Re}}}
\renewcommand{\Im}{\operatorname{\mathfrak{Im}}}
\newcommand{\im}{\mathrm{im}\,}
\newcommand{\inner}[1]{\langle #1 \rangle}
\newcommand{\pair}[2]{\langle #1, #2\rangle}
\newcommand{\ppair}[2]{\langle\langle #1, #2\rangle\rangle}
\newcommand{\Pair}[2]{\left[#1, #2\right]}
\newcommand{\Char}{\mathop{\mathrm{char}}}
\newcommand{\rank}{\mathrm{rank}}
\newcommand{\sgn}[1]{\mathop{\mathrm{sgn}}(#1)}
\newcommand{\leg}[2]{\left(\frac{#1}{#2}\right)}
\newcommand{\Sym}{\mathrm{Sym}}
\newcommand{\hmat}[2]{\left(\begin{array}{cc} #1 & #2\\ -\bar{#2} & \bar{#1}\end{array}\right)}
\newcommand{\HMat}[2]{\left(\begin{array}{cc} #1 & #2\\ -\overline{#2} & \overline{#1}\end{array}\right)}
\newcommand{\Sin}[1]{\sin{\left(#1\right)}}
\newcommand{\Cos}[1]{\cos{\left(#1\right)}}
\newcommand{\comm}[2]{\left[#1, #2\right]}
\newcommand{\Isom}{\mathop{\mathrm{Isom}}}
\newcommand{\Map}{\mathop{\mathrm{Map}}}
\newcommand{\Bij}{\mathop{\mathrm{Bij}}}
\newcommand{\Z}{\mathbb{Z}}
\newcommand{\R}{\mathbb{R}}
\newcommand{\Q}{\mathbb{Q}}
\newcommand{\C}{\mathbb{C}}
\newcommand{\Nm}{\mathrm{Nm}}
\newcommand{\RI}[1]{\mathcal{O}_{#1}}
\newcommand{\F}{\mathbb{F}}
\renewcommand{\Pr}{\displaystyle\mathop{\mathrm{Pr}}\limits}
\newcommand{\E}{\mathbb{E}}
\newcommand{\coker}{\mathop{\mathrm{coker}}}
\newcommand{\id}{\mathop{\mathrm{id}}}
\newcommand{\Oplus}{\displaystyle\bigoplus\limits}
\renewcommand{\Cap}{\displaystyle\bigcap\limits}
\renewcommand{\Cup}{\displaystyle\bigcup\limits}
\newcommand{\Bil}{\mathop{\mathrm{Bil}}}
\newcommand{\N}{\mathbb{N}}
\newcommand{\Aut}{\mathop{\mathrm{Aut}}}
\newcommand{\ord}{\mathop{\mathrm{ord}}}
\newcommand{\ch}{\mathop{\mathrm{char}}}
\newcommand{\minpoly}{\mathop{\mathrm{minpoly}}}
\newcommand{\Spec}{\mathop{\mathrm{Spec}}}
\newcommand{\Gal}{\mathop{\mathrm{Gal}}}
\newcommand{\Ad}{\mathop{\mathrm{Ad}}}
\newcommand{\Stab}{\mathop{\mathrm{Stab}}}
\newcommand{\Norm}{\mathop{\mathrm{Norm}}}
\newcommand{\Orb}{\mathop{\mathrm{Orb}}}
\newcommand{\pfrak}{\mathfrak{p}}
\newcommand{\qfrak}{\mathfrak{q}}
\newcommand{\mfrak}{\mathfrak{m}}
\newcommand{\Frac}{\mathop{\mathrm{Frac}}}
\newcommand{\Loc}{\mathop{\mathrm{Loc}}}
\newcommand{\Sat}{\mathop{\mathrm{Sat}}}
\newcommand{\inj}{\hookrightarrow}
\newcommand{\surj}{\twoheadrightarrow}
\newcommand{\bij}{\leftrightarrow}
\newcommand{\Ind}{\mathrm{Ind}}
\newcommand{\Supp}{\mathop{\mathrm{Supp}}}
\newcommand{\Ass}{\mathop{\mathrm{Ass}}}
\newcommand{\Ann}{\mathop{\mathrm{Ann}}}
\newcommand{\Krulldim}{\dim_{\mathrm{Kr}}}
\newcommand{\Avg}{\mathop{\mathrm{Avg}}}
\newcommand{\innerhom}{\underline{\Hom}}
\newcommand{\triv}{\mathop{\mathrm{triv}}}
\newcommand{\Res}{\mathrm{Res}}
\newcommand{\eval}{\mathop{\mathrm{eval}}}
\newcommand{\MC}{\mathop{\mathrm{MC}}}
\newcommand{\Fun}{\mathop{\mathrm{Fun}}}
\newcommand{\InvFun}{\mathop{\mathrm{InvFun}}}
\renewcommand{\ch}{\mathop{\mathrm{ch}}}
\newcommand{\irrep}{\mathop{\mathrm{Irr}}}
\newcommand{\len}{\mathop{\mathrm{len}}}
\newcommand{\SL}{\mathrm{SL}}
\newcommand{\GL}{\mathrm{GL}}
\newcommand{\PSL}{\mathrm{SL}}
\newcommand{\actson}{\curvearrowright}
\renewcommand{\H}{\mathbb{H}}
\newcommand{\mat}[4]{\left(\begin{array}{cc} #1 & #2\\ #3 & #4\end{array}\right)}
\newcommand{\interior}{\mathop{\mathrm{int}}}
\newcommand{\floor}[1]{\left\lfloor #1\right\rfloor}
\newcommand{\iso}{\cong}
\newcommand{\eps}{\epsilon}
\newcommand{\disc}{\mathrm{disc}}
\newcommand{\Frob}{\mathrm{Frob}}
\newcommand{\charpoly}{\mathrm{charpoly}}
\newcommand{\afrak}{\mathfrak{a}}
\newcommand{\cfrak}{\mathfrak{c}}
\newcommand{\codim}{\mathrm{codim}}
\newcommand{\ffrak}{\mathfrak{f}}
\newcommand{\Pfrak}{\mathfrak{P}}
\newcommand{\homcont}{\hom_{\mathrm{cont}}}
\newcommand{\vol}{\mathrm{vol}}
\newcommand{\ofrak}{\mathfrak{o}}
\newcommand{\A}{\mathbb{A}}
\newcommand{\I}{\mathbb{I}}
\newcommand{\invlim}{\varprojlim}
\newcommand{\dirlim}{\varinjlim}
\renewcommand{\ch}{\mathrm{char}}
\newcommand{\artin}[2]{\left(\frac{#1}{#2}\right)}
\newcommand{\Qfrak}{\mathfrak{Q}}
\newcommand{\ur}[1]{#1^{\mathrm{ur}}}
\newcommand{\absnm}{\mathcal{N}}
\newcommand{\ab}[1]{#1^{\mathrm{ab}}}
\newcommand{\G}{\mathbb{G}}
\newcommand{\dfrak}{\mathfrak{d}}
\newcommand{\Bfrak}{\mathfrak{B}}
\renewcommand{\sgn}{\mathrm{sgn}}
\newcommand{\disjcup}{\bigsqcup}
\newcommand{\zfrak}{\mathfrak{z}}
\renewcommand{\Tr}{\mathrm{Tr}}
\newcommand{\reg}{\mathrm{reg}}
\newcommand{\subgrp}{\leq}
\newcommand{\normal}{\vartriangleleft}
\newcommand{\Dfrak}{\mathfrak{D}}
\newcommand{\nvert}{\nmid}
\newcommand{\K}{\mathbb{K}}
\newcommand{\pt}{\mathrm{pt}}
\newcommand{\RP}{\mathbb{RP}}
\newcommand{\CP}{\mathbb{CP}}
\newcommand{\rk}{\mathrm{rk}}
\newcommand{\redH}{\tilde{H}}
\renewcommand{\H}{\tilde{H}}
\newcommand{\Cyl}{\mathrm{Cyl}}
\newcommand{\T}{\mathbb{T}}
\newcommand{\Ab}{\mathrm{Ab}}
\newcommand{\Vect}{\mathrm{Vect}}
\newcommand{\Top}{\mathrm{Top}}
\newcommand{\Nat}{\mathrm{Nat}}
\newcommand{\inc}{\mathrm{inc}}
\newcommand{\Tor}{\mathrm{Tor}}
\newcommand{\Ext}{\mathrm{Ext}}
\newcommand{\fungrpd}{\pi_{\leq 1}}
\newcommand{\slot}{\mbox{---}}
\newcommand{\funct}{\mathcal}
\newcommand{\Funct}{\mathcal{F}}
\newcommand{\Gunct}{\mathcal{G}}
\newcommand{\FunCat}{\mathrm{Funct}}
\newcommand{\Rep}{\mathrm{Rep}}
\newcommand{\Specm}{\mathrm{Specm}}
\newcommand{\ev}{\mathrm{ev}}
\newcommand{\frpt}[1]{\{#1\}}
\newcommand{\h}{\mathscr{H}}
\newcommand{\poly}{\mathrm{poly}}
\newcommand{\Partial}[1]{\frac{\partial}{\partial #1}}
\newcommand{\Cont}{\mathrm{Cont}}
\renewcommand{\o}{\ofrak}
\newcommand{\bfrak}{\mathfrak{b}}
\newcommand{\Cl}{\mathrm{Cl}}
\newcommand{\ceil}[1]{\lceil #1\rceil}
\newcommand{\hfrak}{\mathfrak{h}}
\newcommand{\Sel}{\mathrm{Sel}}
\newcommand{\Qbar}{\overline{\mathbb{Q}}}
\renewcommand{\I}{\mathrm{I}}
\newcommand{\II}{\mathrm{II}}
\newcommand{\III}{\mathrm{III}}
\newcommand{\IV}{\mathrm{IV}}
\newcommand{\V}{\mathrm{V}}
\newcommand{\FuniversalT}{\mathcal{F}_{\mathrm{universal}}^{\leq T}}
\newcommand{\FAT}{\mathcal{F}_{A=0}^{\leq T}}
\newcommand{\FBT}{\mathcal{F}_{B=0}^{\leq T}}
\newcommand{\FcongT}{\mathcal{F}_{\mathrm{congruent}}^{\leq T}}
\newcommand{\rad}{\mathrm{rad}}
\newcommand{\const}{\mathrm{const}}
\renewcommand{\sp}{\mathrm{span}}
\renewcommand{\d}{\partial}
\newcommand{\num}{\mathrm{num}}
\newcommand{\den}{\mathrm{den}}
\newcommand{\ind}{\mathrm{ind}}

\let\uglyphi\phi
\let\phi\varphi

\tableofcontents

\section{Introduction}
The question of counting the number of integral solutions to an equation of shape $y^2 = x^3 + Ax + B$ goes back at least to Fermat, who, on considering this question for specific $A$ and $B$ (e.g.\ one of his challenge problems to the English was to find all integral solutions to $y^2 = x^3 - 2$), developed his method of descent. Fermat also applied this method to show certain such equations had \emph{no} nontrivial rational solutions (famously, $y^2 = x^3 - x$, showing that $1$ is not the area of a right triangle with rational sides), leading to the question of counting the number of rational solutions to such equations as well.

This last question has seen great progress. Certainly the number of solutions is either infinite or finite, and density considerations (\cite{harronsnowden}) imply that $0\%$ of curves with finitely many rational points have any at all. Recent work of Bhargava-Shankar \cite{bhargavashankarfive} and Bhargava-Skinner-Zhang \cite{bhargavaskinnerzhang} implies that, in fact, both possibilities --- infinitely many and none at all --- occur with positive probability. This agrees with the expectation derived from the Birch and Swinnerton-Dyer conjecture of each possibility occurring with probability one half (the ``minimalist conjecture'' of Goldfeld and Katz-Sarnak).

Progress has also been made for equations of shape $y^2 = f(x)$ with $f\in \Z[x]$ of fixed odd degree $2g+1 > 3$. Here, by Faltings's theorem, one cannot have infinitely many solutions, and indeed one expects none with probability $1$. In fact Poonen-Stoll \cite{poonenstoll}, building on work of Bhargava-Gross \cite{bhargavagross}, were able to prove that such a curve has no rational solutions with probability at least $1 - (12g + 20)2^{-g}$, which is quite close to $1$ for $g$ very large.

But the analogous question for integral points on elliptic curves does not yield to these methods. By a theorem of Siegel there are only finitely many solutions to $y^2 = x^3 + Ax + B$ if $A$ and $B$ are such that the discriminant of the cubic, $-4A^3 - 27B^2$, is nonzero, so that the equation defines an elliptic curve. Therefore we are in a situation like that of Poonen-Stoll/Bhargava-Gross, and similarly we expect to have no integral solutions with probability $1$.\footnote{Indeed, this expectation dates back at least to 1986: see page 269 of the first edition of Silverman's \emph{Arithmetic of Elliptic Curves} \cite{silvermanarithmeticofellipticcurves}.} But despite the expected paucity of curves with integral points, until now it was not known whether the average number of integral points on elliptic curves is bounded. In this paper we show that it is indeed bounded --- in fact, by $66$.

Let us now be more precise. An elliptic curve $E/\Q$ has a unique Weierstrass model of the form $E_{A,B}: y^2 = x^3 + Ax + B$, where $A$ and $B$ are such that $p^4\vert A\implies p^6\nmid B$ and $-4A^3 -27B^2\neq 0$. Given a Weierstrass model, we define the set of integral points on the curve as $$E_{A,B}(\Z) := \{(x,y)\in \Z^2 \vert y^2 = x^3 + Ax + B\},$$ and write $\#|E_{A,B}(\Z)|$ for its cardinality. To produce probabilistic statements, we need a notion of density. We write $H(E_{A,B}) := \max(4|A|^3, 27B^2)^{\frac{1}{6}}$ for the na\"{\i}ve height of $E_{A,B}$. Note that our normalization is slightly different from that of Bhargava-Shankar.

Given a family $\mathcal{F}$ of elliptic curves and a function $f$ on this family, we define $$\Avg_{E\in \mathcal{F}^{\leq T}}(f(E)) := \frac{\Sum_{E\in \mathcal{F}, H(E)\leq T} f(E)}{\Sum_{E\in \mathcal{F}, H(E)\leq T} 1}.$$ Thus for instance Bhargava-Shankar \cite{bhargavashankarfive} have shown that $$\limsup_{T\to\infty} \Avg_{E\in \mathcal{F}^{\leq T}}(5^{\rank(E)})\leq 6$$ for $\mathcal{F}$ the family of all elliptic curves.

Let now $\mathcal{F}_{\mathrm{universal}}$ be the family of all elliptic curves, $\mathcal{F}_{A=0}$ be the family of Mordell curves $y^2 = x^3 + B$ ($B$ sixth-power free), $\mathcal{F}_{B=0}$ be the family of curves $y^2 = x^3 + Ax$ ($A$ fourth-power free), and $\mathcal{F}_{\mathrm{congruent}}$ be the family of congruent number curves $y^2 = x^3 - D^2 x$ ($D$ squarefree). With this notation in hand, we may state our main result.

\begin{thm}\label{amazing theorem}
Let $k\geq 0$. Let $\mathcal{F} = \mathcal{F}_{\mathrm{universal}}, \mathcal{F}_{A=0}, \mathcal{F}_{B=0}$, or $\mathcal{F}_{\mathrm{congruent}}$. \\Then: $$\limsup_{T\to\infty} \Avg_{E\in \mathcal{F}^{\leq T}}(\#|E(\Z)|^k) \leq O(1)^k\cdot \limsup_{T\to\infty}\Avg_{E\in \mathcal{F}^{\leq T}}(3^{k\cdot \rank(E)}),$$ where the implied constant is effective and absolute.
\end{thm}

Work of Bhargava-Shankar \cite{bhargavashankarfive} implies that, for $\mathcal{F} = \mathcal{F}_{\mathrm{universal}}$, $$\limsup_{T\to\infty}\Avg_{E\in \mathcal{F}_{\mathrm{universal}}^{\leq T}}(5^{\rank(E)})\leq 6,$$ whence the right-hand side of the theorem is $\ll 1$ when $k=1$ and indeed when $k\leq \frac{\log{5}}{\log{3}} = 1.4649...$. (Hence e.g.\ the proportion of curves with at least $n$ integral points is $o(n^{-1.4649...})$.) For this family we optimize our bound to get:
\begin{thm}\label{constant theorem}
When all elliptic curves $E/\Q$ are ordered by height, the average number of integral points $\#|E(\Z)|$ is less than $65.8457$.
Moreover, if the minimalist conjecture\footnote{Here by the ``minimalist conjecture'' we mean not only that the ranks of elliptic curves in $\mathcal{F}_{\mathrm{universal}}^{\leq T}$ are distributed $50/50$ between $0$ and $1$ in the limit $T\to\infty$, but also the same statement for the subfamily of $(A,B)\not\equiv (2,2)\pmod{3}$. Otherwise $\frac{8}{9}$ should be replaced by another constant smaller than $1$.} is true, $65.8457$ may be replaced by $\frac{8}{9}$.
\end{thm}\noindent
That is, $$\limsup_{T\to\infty} \Avg_{E\in \mathcal{F}_{\mathrm{universal}}^{\leq T}}(\#|E(\Z)|) < 65.8457,$$ and this upper bound may be replaced by $\leq \frac{8}{9}$ if the minimalist conjecture holds.\footnote{This $\frac{8}{9}$ results from being unable to rule out the possibility of almost every rank one curve having an integral generator in the subfamily $(A,B)\not\equiv (2,2)\pmod{3}$.}

In Section \ref{congruent number curves subsection} we describe how to extend work of Heath-Brown in \cite{heathbrowntwiststwo} to prove that, for $\mathcal{F} = \mathcal{F}_{\mathrm{congruent}}$, $$\limsup_{T\to\infty}\Avg_{E\in \mathcal{F}_{\mathrm{congruent}}^{\leq T}}(k^{\rank(E)})\ll O(1)^{(\log{k})^2}.$$ From this it follows that:
\begin{cor}\label{congruent number theorem}
When the congruent number curves $E: y^2 = x^3 - D^2 x$ ($D\in \Z^+$ squarefree) are ordered by height, the $k$-th moment of the number of integral points $\#|E(\Z)|^k$ is bounded above by $O(1)^{k^2}$, where the implied constant is effective and absolute. In particular, the proportion of curves with at least $n$ integral points decays like $n^{-\Omega(\log{n})}$.
\end{cor}

In Section \ref{kane thorne curves section} we describe how to extend work of Kane \cite{kane} and Kane-Thorne \cite{kanethorne} to prove that, for $\mathcal{F} = \mathcal{F}_{B=0}$, there is a very large (we will quantify this in the proof) subfamily $\widetilde{\mathcal{F}}_{B=0}\subseteq \mathcal{F}_{B=0}$ for which $$\Avg_{E\in \widetilde{\mathcal{F}}_{B=0}^{\leq T}}(k^{\rank(E)})\ll O(1)^{(\log{k})^2}.$$ From this it will follow that:
\begin{cor}\label{B=0 theorem}
When the curves $E: y^2 = x^3 + Ax$ ($A\in \Z^+$ fourth-power free) are ordered by height, the $k$-th moment of the number of integral points $\#|E(\Z)|^k$ is bounded above by $O(1)^{k^2}$, where the implied constant is effective and absolute. In particular, the proportion of curves with at least $n$ integral points decays like $n^{-\Omega(\log{n})}$.
\end{cor}\noindent
The subfamily $\widetilde{\mathcal{F}}_{B=0}$ will essentially be the subfamily determined by the conditions that $A$ be almost squarefree, have a number of prime factors bounded above by a large constant times $\log\log{A}$ (the expected number), and not be a multiple of a modulus supporting a character with a problematic Siegel zero.

Finally, in the case of $\mathcal{F} = \mathcal{F}_{A=0}$, work of Ruth \cite{ruth} bounds the average of $\#|\Sel_2(E)|$, but a bound on the average of $3^{\rank{(E)}}$ is not yet known.\footnote{However, to show boundedness of the average of $\#|E(\Z)|$ over this family, one may proceed as follows. Integral points on $y^2 = x^3 + B$ give solutions to $-(4a^3 + 27b^2) = 108B$ via $a := -3x, b := 2y$, whence also binary cubics $x^3 + axy^2 + by^3$ with discriminant $108B$. Now, by Davenport-Heilbronn, the number of binary cubic forms $f(x,y)$ with discriminant $|\Delta|\ll X$, when taken up to $\GL_2(\Z)$ equivalence, is $\ll X$. (See e.g.\ Theorem 5 of \cite{bhargavashankartsimerman}.) It therefore suffices to show that there are $\ll 1$ many forms $f$ of shape $x^3 + axy^2 + by^3$ in each equivalence class. If an equivalence class has no such forms, then we are done. Otherwise, we need only check that there are $\ll 1$ many $\gamma \in \GL_2(\Z)$ that take $x^3 + axy^2 + by^3$ to a form of shape $x^3 + a'xy^2 + b'y^3$. Note that, given such a $\gamma =: \left(\begin{array}{cc} p & q\\ r & s\end{array}\right)$, $(f\circ \gamma)(1,0) = 1$, so that $f(p,r) = 1$. Moreover, the condition that the $x^2y$ term be zero gives a cubic or linear equation in $q$ depending on whether or not $r=0$ upon imposing $ps - qr = \pm 1$. Hence the number of such $\gamma$ is at most six times the number of solutions of $f(p,r) = 1$ with $p,r\in \Z$. But, by Thue's theorem in the strengthened form of e.g.\ Bennett (who gives an upper bound of $10$) in \cite{bennett}, this is uniformly bounded, completing the argument.}

Having stated our main results, let us now detail the organization of the paper. In Section \ref{notation, previous results, and outline} we set notation, state previous results towards these theorems, and give a detailed argument (leaving inessential details to references to Section \ref{constant theorem proof} along the way) towards Theorem \ref{constant theorem}, proving boundedness by $O(1)$ rather than an explicit constant. We do this because the length of the argument in Section \ref{constant theorem proof} potentially obscures the main ideas, which are already present in the proof of boundedness. In Section \ref{constant theorem proof} we prove Theorem \ref{constant theorem}, leaving the discussion of the optimization of our bounds to Appendix \ref{appendix}. In Section \ref{amazing theorem proof} we then prove Theorem \ref{amazing theorem} for the remaining three families by following the general method used to prove Theorem \ref{constant theorem}. We also prove Corollaries \ref{congruent number theorem} and \ref{B=0 theorem} by adapting the methods of Heath-Brown and Kane-Thorne to control sizes of Selmer groups in these families. Finally, in Appendix \ref{appendix} we provide details of the optimization for Theorem \ref{constant theorem}.

\section{Acknowledgements}

Theorem \ref{constant theorem}, without the explicit constant, was the subject of my senior thesis at Harvard, supervised by Jacob Tsimerman. I would like to thank him and Arul Shankar for suggesting the problem and for all their patience. I would also like to thank Henry Cohn for his help with bounds on spherical codes and for allowing me to use a program he wrote to optimize linear programming upper bounds for codes on $\RP^n$. I would further like to thank Manjul Bhargava, Peter Bruin, Noam Elkies, John Cremona, Roger Heath-Brown, Harald Helfgott, Emmanuel Kowalski, Barry Mazur, Joseph Silverman, Katherine Stange, and Yukihiro Uchida for answering questions related to this work. Finally, I would like to thank Michael Stoll for pointing out a mistake in a previous version of this paper.

\section{Notation, previous results, and outline of the argument}\label{notation, previous results, and outline}

\subsection{Notation}

Let us now set notation. By $f\ll_\theta g$ we will mean that there exists some positive constant $C_\theta > 0$ depending only on $\theta$ such that $|f|\leq C_\theta |g|$ pointwise. If $\theta$ is omitted (i.e.\ we write $f\ll g$), then the implied constant will be absolute. By $f\asymp_\theta g$ we will mean $f\gg_\theta g$ and $f\ll_\theta g$. By $O_\theta(g)$ we will mean a quantity which is $\ll_\theta g$, and by $\Omega_\theta(g)$ we will mean a quantity that is $\gg_\theta g$. By $o(1)$ we will mean a quantity that approaches $0$ in the relevant limit (which will always be unambiguous). By $f = o(g)$ we will mean $f = o(1)\cdot g$, and by $f\asymp g$ we will mean $f = (1 + o(1))g$. We will write $(a,b)$ for the greatest common divisor of two integers $a,b\in \Z$, $\omega(n)$ for the number of prime factors of $n$, $v_p$ for the $p$-adic valuation, $|\cdot|_v$ for the absolute value at a place $v$ of a number field $K$ (normalized so that the product formula holds), and $h(x)$ for the absolute Weil height of $x\in \Qbar$ --- i.e., $$h(x) := \sum_w \frac{[K_w : \Q_v]}{[K : \Q]} \log^+{|x|_w},$$ the sum taken over all places $w$ of $K$ with $v := w\vert_\Q$ and $\log^+(a) := \max(\log{a}, 0)$. Similarly $H(x) := \exp(h(x))$ will denote the multiplicative Weil height of $x\in \Qbar$. Note that, for $\frac{a}{b}\in \Q$ in lowest terms, $H(\frac{a}{b}) = \max(|a|,|b|)$. Given a rational point $P = (x,y)$ on $E_{A,B}: y^2 = x^3 + Ax + B$, $h(P)$ and $H(P)$ will denote $h(x)$ and $H(x)$, respectively. $$\hat{h}(P) := \lim_{n\to\infty} \frac{h(2^n P)}{4^n}$$ will denote the canonical height of $P$, with N\'{e}ron local heights $\hat{\lambda}_v$ such that $$\sum_v \hat{\lambda}_v = \hat{h}.$$ We will similarly write $$\lambda_v(\cdot) := \log^+{|\cdot|_v}.$$ By $\Delta$ or $\Delta_{A,B}$ we will mean $-16(4A^3 + 27B^2)$, the discriminant of $E_{A,B}$. We will write $N_{A,B}$ for the conductor of $E_{A,B}$, defined by $$N_{A,B} = \prod_{p\vert \Delta} p^{e_p},$$ with $e_p = 1$ if $p$ has multiplicative reduction at $p$, and otherwise $e_p\geq 2$ with equality if $p\neq 2,3$. The definitions of $e_2$ and $e_3$ are more complicated, but we will only use that $e_2\leq 8$ and $e_3\leq 5$. By $\psi_n(P)$ we will mean the $n$-th division polynomial of $E_{A,B}$, with zeroes at the nonidentity $n$-torsion points and of homogeneous degree $\frac{n^2 - 1}{2}$ when $x$ is given degree $1$, $y$ degree $\frac{3}{2}$, $A$ degree $2$, and $B$ degree $3$. Note that multiplication by $n$ is then given by $$nP = \left(x(P) - \frac{\psi_{n-1}(P)\psi_{n+1}(P)}{\psi_n(P)^2}, \frac{\psi_{2n}(P)}{2\psi_n(P)^4}\right).$$ In general $\psi_{2n+1}(P)$ is a polynomial of degree $2n^2 + 2n$ in $x,A,B$ with leading coefficient (in $x$) equal to $2n+1$, and $\psi_{2n}(P)$ is $y$ times a polynomial in $x,A,B$ of degree $2n^2 - 2$ with leading coefficient (in $x$) equal to $2n$. By homogeneity, both these polynomials in $x$ have no term of one degree less in $x$ (i.e., they are of the form $c_d x^d + c_{d-2} x^{d-2} + \cdots + c_0$). Finally, we will abuse the word ``average'' to mean ``limsup of the average'' throughout.

\subsection{Previous results}

Now fix $A$ and $B$ for which $\Delta_{A,B}\neq 0$. The first general result bounding integral points on the curve $E_{A,B}$ is Siegel's famous finiteness theorem:
\begin{thm}[Siegel]
$E_{A,B}(\Z)$ is finite.
\end{thm}

Next Baker, as an application of his theory of linear forms in logarithms, gave an effective upper bound on the heights of the integral points on $E_{A,B}$:
\begin{thm}[Baker, \cite{baker}]
Write $H:=H(E_{A,B})$. Let $P\in E_{A,B}(\Z)$. Then: $$|x(P)|\leq e^{(10^7 H)^{10^7}}.$$
\end{thm}\noindent
This of course gives a bound on the number of integral points on $E_{A,B}$.

As in the case of Roth's theorem in Diophantine approximation, effectively bounding the \emph{number} of solutions is much easier than bounding their \emph{heights}. Indeed, Siegel's argument was already effective, and Silverman and Hindry-Silverman were the first to use it to give an explicit upper bound. They obtained:
\begin{thm}[Silverman, \cite{silvermansiegel}]
$$\#|E_{A,B}(\Z)|\ll O(1)^{\rank(E_{A,B}) + \omega(\Delta)}.$$ In fact, one can further reduce $\omega(\Delta)$ to $\omega(\Delta_{ss})$, the number of primes of semistable bad reduction.
\end{thm}

\begin{thm}[Hindry-Silverman, \cite{hindrysilverman}]\label{hindrysilvermanszpirobound}
$$\#|E_{A,B}(\Z)|\ll O(1)^{\rank(E_{A,B}) + \sigma_{E_{A,B}}},$$ where $$\sigma_{E_{A,B}} := \frac{\log{|\Delta_{A,B}|}}{\log{N_{A,B}}}$$ is the Szpiro ratio of $E_{A,B}$ (here $N_{A,B}$ is the conductor of $E_{A,B}$).
\end{thm}

Conjecturally the Szpiro ratio is at most $6+o(1)$. This is equivalent to the ABC conjecture. In any case, the implied constants in both theorems are on the order of $10^{10}$, even if one uses recent improvements to the arguments in Hindry-Silverman (namely, Petsche's \cite{petsche} improved lower bound on the canonical height of a nontorsion rational point on $E_{A,B}$), one cannot reduce the constants to below this order of magnitude. On the other hand it is quite easy to show that most curves have Szpiro ratio at most, say, $100$, so one might think that this makes the second bound amenable to averaging.

But finiteness of the average of $(10^{10})^{\rank(E_{A,B})}$ is far out of the reach of current techniques.\footnote{Heath-Brown \cite{heathbrownsuperexponentialdecay} has proved, assuming the Grand Riemann Hypothesis and the Birch and Swinnerton-Dyer conjecture, that the proportion of curves with rank $R$ is $\ll R^{-\Omega(R)}$, whence we may average $(10^{10})^{\rank(E_{A,B})}$. Thus our result follows from combining this theorem of Heath-Brown with the work of Hindry-Silverman for curves of nonnegligible conductor, and the pointwise bound of Helfgott-Venkatesh (stated below) for those curves of negligible conductor.} Recent spectacular results of Bhargava-Shankar (which will feature centrally in this argument) have proven that the average of $5^{\rank(E_{A,B})}$ is finite (it is at most $6$), and this is the extent of current techniques. Specifically, Bhargava-Shankar have shown:
\begin{thm}[Bhargava-Shankar, \cite{bhargavashankartwo, bhargavashankarthree, bhargavashankarfour, bhargavashankarfive}]
Let $n = 2, 3, 4$, or $5$. Then when all elliptic curves $E/\Q$ are ordered by height, the average size of the $n$-Selmer group $\Sel_n(E)$ is $\sigma(n)$, the sum of divisors of $n$.
\end{thm}\noindent
Note that $n^{\rank(E)}\leq \#|\Sel_n(E)|$ via Galois cohomology, whence the average of $n^{\rank(E)}$ is at most $\sigma(n)$ for $n\leq 5$.

Another result crucial to us is the pointwise bound of Helfgott-Venkatesh, who obtain:
\begin{thm}[Helfgott-Venkatesh, \cite{helfgottvenkatesh}]
$$\#|E_{A,B}(\Z)|\ll O(1)^{\omega(\Delta)}\cdot (\log{|\Delta|})^2\cdot 1.34^{\rank(E_{A,B})}.$$
\end{thm}

From this it follows that (see Lemma \ref{using the pointwise bound}):
\begin{cor}\label{literature bound}
$$\Avg_{E\in \mathcal{F}_{\mathrm{universal}}^{\leq T}}(\#|E(\Z)|)\ll_\eps T^\eps.$$
\end{cor}\noindent
To the author's knowledge, except for a potentially small improvement (e.g.\ $\exp\left(O\left(\frac{\log{T}}{\log\log{T}}\right)\right)$ instead of $T^\eps$), this is the best result derivable directly from the literature in this direction.\footnote{There has been extensive work by Heath-Brown \cite{heathbrowncurvesandsurfaces}, Bombieri-Pila \cite{bombieripila}, and others on bounding the number of rational points of small height, but this does not improve the above bound.} This sort of result will allow us to restrict our attention to subfamilies of density $1 - T^{-\Omega(1)}$, which will be quite useful in what follows.

\subsection{Detailed sketch of proof of boundedness}

Let us now give an argument proving Theorem \ref{constant theorem} without an explicit constant. (To lower the constant to $66$ we will have to be much more careful.)
\begin{proof}[Sketch of proof that $\limsup_{T\to \infty}\Avg_{E\in \mathcal{F}_{\mathrm{universal}}^{\leq T}}(\#|E(\Z)|) < \infty$.]
The first thing to note is that the size of $\mathcal{F}_{\mathrm{universal}}^{\leq T}$ is $\asymp T^5$. (Indeed, the bound $H(E_{A,B})\ll T$ is equivalent to the bounds $A\ll T^2$ and $B\ll T^3$.)

By Corollary \ref{literature bound}, we may restrict to any subfamily of density at least $1 - T^{-\Omega(1)}$.\footnote{See Lemma \ref{using the pointwise bound}.} Fix a $\delta > 0$. We will restrict to the subfamily $\mathcal{F}_*\subseteq \mathcal{F}_{\mathrm{universal}}$ with:
\begin{itemize}
\item $|A|\gg T^{2-\delta}, |B|\gg T^{3-\delta}$.
\item $(A,B)\leq T^\delta$.
\item $\prod_{v_p(\Delta)\geq 2} p^{v_p(\Delta)}\leq T^\delta$.\footnote{To see that this has the desired density, see Lemma \ref{getting rid of annoying reduction}.}
\end{itemize}

On this subfamily we break the integral points into three classes:
$$E(\Z) = E(\Z)_{\textrm{small}} \cup E(\Z)_{\textrm{medium}} \cup E(\Z)_{\textrm{large}},$$ where:
\begin{align*}
E(\Z)_{\textrm{small}} &:= \{P\in E(\Z) \vert h(P)\leq (5-\delta)\log{T}\},\\
E(\Z)_{\textrm{medium}} &:= \{P\in E(\Z) \vert (5-\delta)\log{T} < h(P)\leq \delta^{-1}\log{T}\},\\
E(\Z)_{\textrm{large}} &:= \{P\in E(\Z) \vert \delta^{-1}\log{T} < h(P)\}.
\end{align*}\noindent
We will call these the ``small'', ``medium'', and ``large'' ranges, respectively.

By explicit counting, we obtain the bound $\sum_{A\ll T^2, B\ll T^3} \#|E_{A,B}(\Z)_{\textrm{small}}|\ll T^{5-\delta}$.\footnote{See the proof of the second part of Lemma \ref{throwing out small points}.} Therefore the small range does not contribute to the average.

To bound the points in the medium range, we prove a gap principle (analogous to the Mumford gap principle for rational points on higher genus curves) which seems to have first appeared in work of Silverman \cite{silvermansiegel} and Helfgott \cite{helfgottsquarefree}.\footnote{The difficulty in proving this in fact lies in handling the error term, which relies in a careful estimation of the difference between the Weil and canonical height on this curve. (This is the reason for restricting to the subfamily $\mathcal{F}_*$: the difference between the two heights is much better controlled in this case.) See Lemma \ref{lower bound on angle}.}
\begin{lem}[Helfgott-Mumford gap principle]
Let $P,R\in E(\Z)_{\textrm{medium}}\cup E(\Z)_{\textrm{large}}$. Let $\theta_{P,R}$ be the angle between them in the Mordell-Weil lattice $E(\Q)/\mathrm{tors}\subseteq E(\Q)\otimes_\Z \R$ (with respect to the canonical height). Then: $$\cos{\theta_{P,R}}\leq \frac{1}{2}\max\left(\sqrt{\frac{h(P)}{h(R)}}, \sqrt{\frac{h(R)}{h(P)}}\right) + O(\delta).$$
\end{lem}

Therefore, via $P\mapsto \frac{1}{\sqrt{\hat{h}(P)}}\otimes P$, the number of points $P\in E(\Z)_{\textrm{medium}}$ with canonical height in the range $[X, (1+\delta)X]$ is $$\ll A(\rank(E),\theta_0),$$ where $\theta_0 = \frac{\pi}{3} - O(\delta)$, and $A(n,\theta)$ is the maximal number of unit vectors in $\R^n$ with pairwise angles at least $\theta$. It is a well-known problem in the theory of sphere packing to provide a good upper bound for this quantity. For our purposes we will be interested in an upper bound for large $n$, and one is provided by the work of Kabatiansky-Levenshtein:
\begin{thm}[Kabatiansky-Levenshtein, \cite{kabatianskylevenshtein}]\label{kablev}
$$A(n,\theta)\ll \exp\left(n\cdot \left[\frac{1+\sin{\theta}}{2\sin{\theta}}\log{\left(\frac{1+\sin{\theta}}{2\sin{\theta}}\right)} - \frac{1-\sin{\theta}}{2\sin{\theta}}\log{\left(\frac{1-\sin{\theta}}{2\sin{\theta}}\right)} + o(1)\right]\right).$$
\end{thm}\noindent
For $\theta_0 = \frac{\pi}{3} - O(\delta)$, this tells us that $$A(n,\theta_0)\ll 1.33^n$$ once $\delta\ll 1$.

Therefore the number of integral points with canonical height in the interval $[X,(1+\delta)X]$ is $\ll 1.33^{\rank(E)}$. Since we can cover $E(\Z)_{\textrm{medium}}$ with $O(\delta^{-1})$ such intervals, we obtain the bound $$\#|E(\Z)_{\textrm{medium}}|\ll \delta^{-1}\cdot 1.33^{\rank(E)}.$$ Since, by Bhargava-Shankar, the average of $2^{\rank(E)}$ is bounded over this family, the medium range contributes $O(\delta^{-1})$ to the average.

Finally, to the large range. The claim is that there are $O(\delta^{-1}\log{(\delta^{-1})}\cdot 1.33^{\rank(E)})$ many points of $E(\Z)_{\textrm{large}}$ in each coset of $E(\Q)/3E(\Q)$. To see this, let $R$ be a minimal element (with respect to height) of $E(\Z)_{\textrm{large}}$ in its coset modulo $3$. By the same argument as for the medium range, there are $O(\delta^{-1}\log{(\delta^{-1})}\cdot 1.33^{\rank(E)})$ integral points $P$ with $h(P) < \delta^{-1}h(R)$. For those points $P\equiv R\pmod{3}$ with $h(P)\geq \delta^{-1}h(R)$, we write $P =: 3Q + R$ with $Q\in E(\Q)$. Then since $P$ is very close to $\infty$ in the Archimedean topology, $Q$ must be very close to a solution of $3\tilde{R} = -R$ as well. That is, $x(Q)$ must be very close to an $x(\tilde{R})\in \Qbar$ solving $x(3\tilde{R}) = x(R)$. After making this precise\footnote{See \eqref{end of calculation of kappa} and take $C, D\gg \delta^{-1}$.}, we find that:
$$\frac{9}{2} - O(\delta)\geq \frac{\log{|x(Q) - x(\tilde{R})|^{-1}}}{h(Q)}$$ for some such $\tilde{R}$. Therefore $$|x(Q) - x(\tilde{R})|\leq H(Q)^{-\frac{9}{2} + O(\delta)}.$$ Thus $x(Q)$ is a Roth-type approximation to $x(\tilde{R})$. Moreover since $h(x(Q)) = h(Q)\gg \delta^{-1} h(\tilde{R}) = \delta^{-1} h(x(\tilde{R}))$, we see that $x(Q)$ is a ``large'' rational approximation, in the sense of Bombieri-Gubler \cite{bombierigubler}. As they prove\footnote{See e.g.\ their (6.23).}, there are only $O(1)$ such approximations once $\delta^{-1}\gg 1$. Therefore each coset modulo $3$ contributes at most $O(\delta^{-1}\log{(\delta^{-1})}\cdot 1.33^{\rank(E)})$ to $\#|E(\Z)_{\textrm{large}}|$, whence we obtain the bound $$\#|E(\Z)_{\textrm{large}}|\ll \delta^{-1}\log{(\delta^{-1})}\cdot 3.99^{\rank(E)}.$$ Again by Bhargava-Shankar the average of $4^{\rank(E)}$ is bounded, so that the large range contributes $O(\delta^{-1}\log{(\delta^{-1})})$ to the average.

Therefore, in sum, we have found that the average is at most $O(\delta^{-1}\log{(\delta^{-1})})$ for any $\delta\ll 1$ sufficiently small. Choosing such a $\delta\asymp 1$ then gives the result.
\end{proof}

Having given a sketch of an argument proving the weaker theorem that the limsup of the average is bounded, let us now give the full proof of Theorem \ref{constant theorem}.

\section{Proof of Theorem \ref{constant theorem}}\label{constant theorem proof}

We will follow the structure of the argument given in the previous section reasonably closely, deviating only in the specific details of the application of sphere-packing bounds (for numerical reasons), and in being entirely explicit. We work only with the average (i.e., $k=1$) --- there are only a few modifications required for the case of general $k$, and they are all clear.

\begin{proof}[Proof of Theorem \ref{constant theorem}.]
As noted, $$\#|\mathcal{F}_{\mathrm{universal}}^{\leq T}|\asymp T^5.$$

To obtain a good estimate on the difference between the Weil height and the canonical height, we will restrict to a subfamily $\mathcal{F}_*\subseteq \mathcal{F}_{\mathrm{universal}}^{\leq T}$ which omits a set of density $O(T^{-c})$ for some positive $c > 0$. The following lemma shows that we may do this.

\subsection{Restricting to a subfamily and handling small points}

\begin{lem}\label{using the pointwise bound}
Let $\mathcal{G}\subseteq \mathcal{F}_{\mathrm{universal}}^{\leq T}$. Then, for all $\eps > 0$, $$\sum_{E\in \mathcal{G}} \#|E(\Z)|\ll \#|\mathcal{F}_{\mathrm{universal}}^{\leq T}|\cdot \left(\frac{\#|\mathcal{G}|}{\#|\mathcal{F}_{\mathrm{universal}}^{\leq T}|}\right)^{\Omega(1)}\cdot \exp\left(O\left(\frac{\log{T}}{\log\log{T}}\right)\right).$$
\end{lem}
\begin{proof}
By H\"{o}lder's inequality, it suffices to show that $$\Avg_{E\in \mathcal{F}_{\mathrm{universal}}^{\leq T}}(\#|E(\Z)|^{1.0001})\ll \exp\left(O\left(\frac{\log{T}}{\log\log{T}}\right)\right).$$ By Helfgott-Venkatesh (Theorem \ref{literature bound}), we have that $$\sum_{E\in \mathcal{F}_{\mathrm{universal}}^{\leq T}} \#|E(\Z)|^{1.0001}\ll (\log{T})^{2.0002}\cdot \sum_{E\in \mathcal{F}_{\mathrm{universal}}^{\leq T}} O(1)^{\omega(\Delta_E)}\cdot 1.35^{\rank(E)}.$$ We apply the crude bound $\omega(n)\ll \frac{\log{n}}{\log\log{n}}$ and Bhargava-Shankar to conclude.
\end{proof}

Fix a $\delta > 0$ to be chosen later. We will take $\delta\asymp 1$ independent of $T$. Let us apply this to first restrict to the subfamily $\mathcal{F}_\bullet\subseteq \FuniversalT$ defined by the conditions:
\begin{enumerate}
\item $|A|\geq T^{2-\delta}$.
\item $|B|\geq T^{3-\delta}$, and $B$ is not a square.
\item $(A,B)\leq T^\delta$.
\item $|\Delta|\geq T^{6-2\delta}$.
\item $\prod_{p^2\vert \Delta} p^{v_p(\Delta)}\leq T^{4\delta}$.
\end{enumerate}

To see that we may, we prove:
\begin{lem}\label{getting rid of annoying reduction}
Let $\mathcal{G}$ be the complement of $\mathcal{F}_\bullet$ in $\FuniversalT$. Then: $$\frac{\#|\mathcal{G}|}{\#|\FuniversalT|}\ll T^{-\Omega(\delta)}.$$
\end{lem}
\begin{proof}
It suffices to impose each condition one by one and check that we throw out a density $\ll T^{-\Omega(\delta)}$ subset at each step. For the first and second conditions this is immediate. For the third condition, the number of $A\ll T^2, B\ll T^3$ with $(A,B) > T^\delta$ is at most $$\ll \sum_{T^\delta < n\ll T^2} \frac{T^2}{n}\cdot \frac{T^3}{n}\ll T^{5-\delta}.$$ So we may assume the first, second, and third conditions. Given these, for the fourth condition, if $|\Delta| < T^{6-2\delta}$, then
\begin{align*}
A &= \left(-\frac{27}{4}B^2 + O(T^{6-2\delta})\right)^{\frac{1}{3}}
\\&= -\frac{3}{2^{\frac{2}{3}}} B^{\frac{2}{3}} + O\left(\frac{T^{6-2\delta}}{B^{\frac{4}{3}}}\right)
\\&= -\frac{3}{2^{\frac{2}{3}}} B^{\frac{2}{3}} + O\left(T^{2-\frac{\delta}{3}}\right).
\end{align*}
Therefore the number of $A,B$ with $|\Delta_{A,B}| < T^{6-2\delta}$ is $$\ll \sum_{B\ll T^3} T^{2-\frac{\delta}{3}}\ll T^{5-\frac{\delta}{3}}.$$

Finally, for the fifth condition given the other four, the argument will be a bit longer. Our strategy will be to show that we may take the radical of $\Delta$ to be reasonably large, and then we will establish that we may take $\Delta$ to not have any nonnegligible square divisors. Then we may bound the ``nonsquarefree part'' of $\Delta$ in terms of square divisors of $\Delta$ only, which thus forces it to be small.

We first show that we may assume the conductor of $E_{A,B}$ is at least $T^{4.08}$. To see this, by Theorem 4.5 of Helfgott-Venkatesh \cite{helfgottvenkatesh}, the number of curves of conductor $N$ is $\ll N^{0.224}$. Therefore the number of $(A,B)$ with conductor at most $T^{4.08}$ is $\ll T^{1.224\cdot 4.08} < T^{4.999}$, giving the claim.

Now note that $E_{A,B}$ has additive reduction at $p$ if and only if $p\vert (A,B)$. Therefore $$N_{A,B}\ll \left(\prod_{p\neq 2,3, p\vert \Delta} p\right)\cdot \left(\prod_{p\neq 2,3, p\vert (A,B)} p\right)\ll \rad(\Delta)\cdot T^\delta,$$ where $\rad(n) := \prod_{p\vert n} p$ is the radical of $n$. Therefore $\rad(\Delta)\gg T^{4.05}$ once $\delta\ll 1$.

Let us now show that we may assume that if $n^2\vert \Delta$ and $n\ll T^{1.99}$ then $n\leq T^\delta$. To see this, note that $n^2\vert \Delta$ implies that $-64A^3\equiv 432B^2\pmod{n^2}$. The first claim is that, for fixed $A$, the number of $B\ll T^3$ solving this equation modulo $n^2$ is $$\ll O(1)^{\omega(n)}\cdot (A^{\frac{3}{2}},n)\cdot \left(1 + \frac{T^3}{n^2}\right),$$ where $(x^\alpha, y^\beta) := \prod_{p\vert (x,y)} p^{\max(\alpha v_p(x), \beta v_p(y))}$. Indeed, at a prime power $p^e$ with $p > 3$, the number of square roots of $-\frac{4A^3}{27}$ is at most $2p^{\frac{3v_p(A)}{2}}$ by Hensel's lemma. When $p = 3$ it is instead at most $\ll 3^{\frac{3v_3(A)}{2}}$ for the same reason, but the implied constant is different. Similarly at $p = 2$ it is $\ll 2^{\frac{3v_2(A)}{2}}$. Moreover if $3v_p(A)\geq e$, then the number of solutions for $B$ is instead at most $\const\cdot p^{\frac{e}{2}}$, with $\const\ll 1$ and equal to $1$ if $p > 3$. Therefore the number of solutions modulo $m$ is $$\ll O(1)^{\omega(m)}\cdot \left(\prod_{p\vert (A,m)} p^{\min\left(v_p(A),\frac{1}{3}v_p(m)\right)}\right)^{\frac{3}{2}}.$$ Hence the number of $B\ll T^3$ such that $n^2\vert \Delta_{A,B}$ is $$\ll O(1)^{\omega(n)}\cdot (A^{\frac{3}{2}},n)\cdot \left(1 + \frac{T^3}{n^2}\right).$$

But then the number of $A\ll T^2, B\ll T^3$ for which there exists an $n^2\vert \Delta$ with $T^\delta < n\ll T^{1.99}$ is at most
\begin{align*}
\sum_{A\ll T^2}\sum_{B\ll T^3}\sum_{T^\delta < n\ll T^{1.99}, n^2\vert \Delta_{A,B}} 1 &= \sum_{T^\delta < n\ll T^{1.99}}\sum_{A\ll T^2}\#|\{B\ll T^3 : n^2\vert \Delta_{A,B}\}|\\&\ll \sum_{T^\delta < n\ll T^{1.99}} O(1)^{\omega(n)}\left(1 + \frac{T^3}{n^2}\right)\sum_{A\ll T^2} (A^{\frac{3}{2}},n).
\end{align*}\noindent
By examining the residue of the relevant Dirichlet series at $s=1$, one finds that $$\sum_{A\ll T^2} (A^{\frac{3}{2}},n)\ll O(1)^{\omega(n)}\cdot \left(\prod_{p^2\vert n} p^{v_p(n)}\right)^{\frac{1}{3}}\cdot T^2.$$ We will again use the bound $O(1)^{\omega(n)}\ll_\eps n^\eps$ to conclude that our sum is at most
\begin{align*}
&\ll_\eps T^{2+\eps}\cdot \sum_{T^\delta < n\ll T^{1.99}} \left(\prod_{p^2\vert n} p^{\frac{v_p(n)}{3}}\right)\cdot \left(1 + \frac{T^3}{n^2}\right)
\\&\ll T^{3.99+\eps} + T^{5 - \frac{2\delta}{3}},
\end{align*}\noindent
as desired.

Therefore we may assume that the only square divisors $n^2$ of $\Delta$ with $n\ll T^{1.99}$ are smaller than $T^{2\delta}$. But now $\left(\prod_{p^2\vert \Delta} p\right)^2$ and $\left(\prod_{p^2\vert \Delta} p^{\floor{\frac{v_p(\Delta)}{2}}}\right)^2$ are square divisors of $\Delta$. Moreover, $\prod_{p^2\vert \Delta} p$ and $\prod_{p^2\vert \Delta} p^{\floor{\frac{v_p(\Delta)}{2}}}$ divide $\frac{\Delta}{\rad(\Delta)}\ll T^{1.95}$. Therefore these square divisors must both be of size at most $T^{2\delta}$! Hence since $\prod_{p^2\vert \Delta} p^{v_p(\Delta)}$ divides $$\left(\prod_{p^2\vert \Delta} p\right)^2\cdot \left(\prod_{p^2\vert \Delta} p^{\floor{\frac{v_p(\Delta)}{2}}}\right)^2\leq T^{4\delta},$$ we are done.
\end{proof}

We will further restrict to a subfamily of curves with no small integral or rational points. Specifically, let $\mathcal{F}_*\subseteq \mathcal{F}_\bullet$ be the subfamily defined by the conditions:
\begin{enumerate}
\item $E_{A,B}(\Q)_{\mathrm{tors}} = 0$.
\item If $P\in E_{A,B}(\Z)$, then $h(P) > (5-\delta)\log{T}$.
\item If $Q\in E_{A,B}(\Q)$, then $h(Q) > \left(\frac{1}{2}-\delta\right)\log{T}$.
\end{enumerate}

Let us now prove that we may restrict to this subfamily.
\begin{lem}\label{throwing out small points}
Let $\mathcal{G}$ be the complement of $\mathcal{F}_*$ in $\mathcal{F}_\bullet$. Then $$\frac{\#|\mathcal{G}|}{\#|\mathcal{F}_\bullet|}\ll T^{-\Omega(\delta)}.$$
\end{lem}
\begin{proof}
Theorem 1.1 in Harron-Snowden \cite{harronsnowden} allows us to impose the first condition. For the second condition, the number of $A\ll T^2, B\ll T^3$ such that there is at least one integral point $P\in E_{A,B}(\Z)$ with $h(P)\leq (5-\delta)\log{T}$ is at most
\begin{align*}
&\#|\{(x,y,A,B)\in \Z^4 : |x|\leq T^{5-\delta}, A\ll T^2, B\ll T^3, y^2 = x^3 + Ax + B\}| \\\quad\quad&= \#|\{(x,y,A,B) : |x|\leq 10^{10}T, A\ll T^2, B\ll T^3, y^2 = x^3 + Ax + B\}| \\&\quad\quad\quad+ \sum_{10^{10}T\leq |x|\ll T^{5-\delta}} \#|\{(y,A,B) : A\ll T^2, B\ll T^3, y^2 = x^3 + Ax + B\}|.
\end{align*}\noindent
To bound the first sum, note that, given $(x,y,A)$, $B = y^2 - x^3 - Ax$ is determined. Moreover $$y^2\ll |x|^3 + T^2 |x| + T^3\ll T^3,$$ so that $y\ll T^{\frac{3}{2}}$. Therefore the number of $(x,y,A,B)$ is at most $$\ll T\cdot T^{\frac{3}{2}}\cdot T^2 = T^{4.5}.$$

For the second sum, note that in this range $$|y^2 - x^3|\ll T^2 |x| + T^3\ll T^2 |x|,$$ whence $y\asymp |x|^{\frac{3}{2}}$. Now, if $(y,A,B)$ and $(y',A',B')$ lie in the solution set and (without loss of generality) $y,y' > 0$, then $$y^2 - y'^2 = x(A-A') + (B-B'),$$ so that $$|y - y'|\ll \frac{T^2|x| + T^3}{|x|^{\frac{3}{2}}}\ll \frac{T^2}{|x|^{\frac{1}{2}}}.$$ Therefore the number of $y$ for which there exist $A,B$ making $(x,y,A,B)$ a solution is at most $$\ll 1 + \frac{T^2}{|x|^{\frac{1}{2}}}.$$

Next, given $x$ and $y$, if $(x,y,A,B)$ and $(x,y,A',B')$ are solutions, then $(A-A')x = B'-B$, so that $$|A - A'|\ll \frac{T^3}{|x|},$$ whence the number of $A$ for which there exists a $B$ making $(x,y,A,B)$ a solution is at most $$\ll 1 + \frac{T^3}{|x|}.$$

Putting these together, the second sum is bounded above by
\begin{align*}
&\sum_{10^{10}T\leq |x|\ll T^{5-\delta}} \#|\{(y,A,B) : A\ll T^2, B\ll T^3, y^2 = x^3 + Ax + B\}|\\\quad\quad&\ll \sum_{10^{10}T\leq |x|\ll T^{5-\delta}} \left(1 + \frac{T^2}{|x|^{\frac{1}{2}}}\right)\left(1 + \frac{T^3}{|x|}\right)\\\quad\quad&\ll T^{5-\delta},
\end{align*}\noindent
as desired.

Finally, for the third condition, note, as above, that the number of $A\ll T^2, B\ll T^3$ such that there is at least one rational point $Q = \left(\frac{x}{d^2}, \frac{y}{d^3}\right)\in E_{A,B}(\Q)$ with $h(Q)\leq \left(\frac{1}{2}-\delta\right)\log{T}$ is at most
\begin{align*}
\#|(x,y,d,A,B) : y^2 = x^3 + Ad^4 x + Bd^6, |x|\leq T^{\frac{1}{2}-\delta}, |d|\leq T^{\frac{1}{4}-\frac{\delta}{2}}, A\ll T^2, B\ll T^3\}.
\end{align*}
Note that if $(x,y,d,A,B)$ is a solution, then $y\ll T^{\frac{3}{2}} d^3$. Moreover, $(x,y,d,A)$ determines $B$. Hence this count is at most: $$\ll T^{\frac{1}{2}-\delta}\cdot \left(T^{\frac{3}{2}}\cdot T^{\frac{3}{4} - \frac{3\delta}{2}}\right)\cdot T^{\frac{1}{4}-\frac{\delta}{2}}\cdot T^2 = T^{5 - 3\delta},$$ whence we are done.
\end{proof}

\subsection{Local heights and a gap principle}

The purpose of restricting to this subfamily is to be able to give a very strong estimate on the difference between the Weil and canonical heights on the curves in this family. Specifically,
\begin{lem}\label{weil vs canonical}
Let $E\in \mathcal{F}_*$. Let $h,\hat{h}$ be the Weil and canonical heights on $E_{A,B}$, respectively. Let $Q\in E(\Q)$. Then $$\hat{h}(Q) - h(Q) = \log^+{|\Delta^{-\frac{1}{6}}x(Q)|} + \frac{1}{6}\log{|\Delta|} - \log^+{|x(Q)|} + O(\delta\log{T}).$$ In particular, $$h(Q)\leq \hat{h}(Q) + O(\delta\log{T})$$ and, if $|x(Q)|\geq |\Delta|^{\frac{1}{6}}$, $$\hat{h}(Q) - h(Q) = O(\delta\log{T}).$$
\end{lem}
\begin{proof}
Write $\hat{h} - h = \sum_v \hat{\lambda}_v - \lambda_v$, where $\lambda_v := \log^+{|\cdot|_v}$, $\hat{\lambda}_v$ are the N\'{e}ron local heights, and $v$ runs over the places of $\Q$. At a prime $p\neq 2,3$ of good reduction, by e.g.\ Theorem 4.1\footnote{Note: our normalization differs from Silverman's by a factor of $2$.} in \cite{silvermandifferencemathcomp}, the local heights are equal. At a prime $p$ of additive reduction (so $p\vert (A,B)$) or at $p = 2$ or $3$, by the same theorem we see that $$0\leq \hat{\lambda}_p - \lambda_p\leq -\frac{1}{6}\log{|\Delta|_p}.$$ Since $p\vert (A,B)$ implies $p^2\vert \Delta$, we see that $$\prod_{p\vert 6(A,B)} p^{v_p(\Delta)}\leq 6\prod_{p^2\vert \Delta} p^{v_p(\Delta)}\ll T^{4\delta},$$ whence the sum of these contributions is $$0\leq \sum_{p\vert 6(A,B)} \hat{\lambda}_p - \lambda_p\ll \delta\log{T}.$$ At a prime $p\neq 2,3$ of multiplicative reduction, by Chapter III Theorem 5.1 of \cite{lang}, since $v_p(\Delta) = 1$ (whence $\alpha = 0$ in Lang's notation), we see that $$\hat{\lambda}_p - \lambda_p = -\frac{1}{6}\log{|\Delta|_p}.$$ Finally, at the infinite place, since $j(E_{A,B})\ll T^{O(\delta)}$, by combining Proposition 5.4 and (31) of \cite{silvermandifferencemathcomp} we find that $$\hat{\lambda}_\infty(Q) - \lambda_\infty(Q) = \log^+{|\Delta^{-\frac{1}{6}}x(Q)|} - \log^+{|x(Q)|} + O(\delta\log{T}).$$ Summing these all up and using the product formula gives the result.
\end{proof}

Given that the Weil and canonical heights are so close, we may now prove a bound on the angle between two integral points by proving a corresponding bound with Weil heights replacing canonical heights. Specifically,
\begin{lem}[Helfgott-Mumford gap principle, cf.\ \cite{helfgottsquarefree}]\label{helfgott-mumford gap principle}
Let $E\in \mathcal{F}_*$. Let $P\neq R\in E(\Z)$ with $h(P)\geq h(R)$ (recall that automatically $h(P), h(R) > (5-\delta)\log{T}$). Then: $$\hat{h}(P+R)\leq 2h(P) + h(R) + O(1).$$
\end{lem}
\begin{proof}
Write $P =: (X,Y)$ and $R =: (x,y)$ with $|X|\geq |x|$. Note that since $|X|, |x|\geq T^{5-\delta}$, we have that $|Y|\sim |X|^{\frac{3}{2}}$ and $|y|\sim |x|^{\frac{3}{2}}$. Now
\begin{align*}
x(P+R) &= \frac{(Y-y)^2}{(X-x)^2} - X - x
\\&= \frac{X^2 x + X x^2 - 2Yy + A(X + x) + 2B}{(X - x)^2}.
\end{align*}
The numerator has absolute value at most $\ll |X|^2 |x|$ by hypothesis. The denominator has absolute value at most $\ll |X|^2$. Therefore, since cancelling common factors will only make the numerator and denominator smaller, we see that $h(P+R)\leq 2h(P) + h(R) + O(1)$. If $|x(P+R)|\geq |\Delta|^{\frac{1}{6}}$, then this completes the proof, by Lemma \ref{weil vs canonical}. Otherwise, write $x(P+Q) = \frac{W}{Z}$ in lowest terms. Then \begin{align*}
\hat{h}(P+R) &= h(P+R) + \frac{1}{6}\log{|\Delta|} - \log^+{|x(P+R)|} + O(\delta\log{T})
\\&= \max(\log{|W|}, \log{|Z|}) + \log{T} - \max(\log{|W|} - \log{|Z|}, 0) + O(\delta\log{T})
\\&= \log{T} + \log{|Z|} + O(\delta\log{T}).
\end{align*}
Since as we saw $|Z|\ll |X|^2$, we find that $\hat{h}(P+R)\leq \log{T} + 2h(R) + O(\delta\log{T})$. Observing that $h(P)\geq (5-\delta)\log{T}$ finishes the result.
\end{proof}\noindent

This results in a lower bound on the angle of integral points close in absolute value:
\begin{lem}\label{lower bound on angle}
Let $E\in \mathcal{F}_*$. Let $P\neq R\in E(\Z)$. Let $\theta_{P,R}$ be the angle between $P$ and $R$ in the Euclidean space $E(\Q)\otimes_\Z \R$. Then: $$\cos{\theta_{P,R}}\leq \frac{1}{2}\max\left(\sqrt{\frac{h(P)}{h(R)}}, \sqrt{\frac{h(R)}{h(P)}}\right) + O(\delta).$$
\end{lem}
\begin{proof}
By definition, $$\cos{\theta_{P,R}} = \frac{\hat{h}(P+R) - \hat{h}(P) - \hat{h}(R)}{2\sqrt{\hat{h}(P)\hat{h}(R)}}.$$ By Lemma \ref{weil vs canonical} and the fact that $h(P), h(R) > (5-\delta)\log{T}$, we find that $$\cos{\theta_{P,R}} = \frac{\hat{h}(P+R) - h(P) - h(R)}{2\sqrt{h(P)h(R)}} + O(\delta).$$ Applying Lemma \ref{helfgott-mumford gap principle} then concludes the argument.
\end{proof}

\subsection{Decomposing the set of integral points into classes: $\I$--$\IV$}

Fix now a parameter $D > 1$. We will take $D$ to be $\ll 1$ in the end. Let $$\tilde{D} := \frac{D + \sqrt{D^2 + 4}}{2},$$ so that \begin{align}\label{def of tilde D}\frac{\tilde{D}^2}{(\tilde{D}^2-1)^2} = \frac{1}{D^2}.\end{align} Fix $E\in \mathcal{F}_*$. Let $r := \rank(E)$. Note that we may assume $r > 0$ since $E(\Q)_{\mathrm{tors}} = 0$ and so $\#|E(\Z)| = 0$ if $r = 0$. So choose $P_1, \ldots, P_r\in E(\Q)$ such that $P_1\neq 0$ has minimal canonical height (recall that $E$ has no rational torsion) and $P_i$ has minimal canonical height among points not inside $\sp_\Z(P_1, \ldots, P_{i-1}) + 3E(\Q)$. Note that since $$\hat{h}(P_i \pm P_j)\geq \hat{h}(P_{\max(i,j)})$$ it follows that $$|\langle P_i, P_j\rangle|\leq \frac{\hat{h}(P_{\min(i,j)})}{2}.$$ It follows that, for any $\eps_i = \pm 1$,
\begin{align}\label{bound on Rs}
\hat{h}\left(\sum_{i=1}^k \eps_i P_i\right)\leq \sum_{i=1}^k (k-i+1)\hat{h}(P_i).
\end{align}

Next note that $P_1, \ldots, P_r$ is an $\F_3$-basis for $E(\Q)/3E(\Q)$. Given $Q\in E(\Q)$, write $i(Q) := \min \{i\vert Q\in \sp_\Z(P_1, \ldots, P_i) + 3E(\Q)\}$ --- i.e., $i(Q)$ is the least $i$ for which $Q$ is congruent to an element of the $\Z$-span of $P_1, \ldots, P_i$ modulo $3$. (Note that $i = 0$ implies $Q$ is a multiple of $3$.) Write $$H_1 := \max\left((5-O(\delta))\log{T}, \hat{h}(P_1)\right),$$ where, say, the implied constant is larger than one plus twice the implied constants in Lemma \ref{weil vs canonical}, and $$H_i := \max\left(\hat{h}(P_i), \tilde{D}^2\cdot H_{i-1}\right).\footnote{For instance, the condition $h(P) > \tilde{D}^2\cdot H_i$ implies $h(P) > \tilde{D}^{2(i-j+1)} \hat{h}(P_j)$ for every $j\leq i$, and it also implies $h(P) > \tilde{D}^{2i}\cdot (5-O(\delta))\log{T}$.}$$ Then if $r > 1$ write
\begin{align*}
E(\Z) &= 3E(\Q)\cap E(\Z)\\&\cup \bigcup_{i = 1}^r \{P\in E(\Z), H_i\leq \hat{h}(P)\leq \tilde{D}^2\cdot H_i\}\\&\cup \bigcup_{i=1}^r \{P\in E(\Z), i(P) = i, \hat{h}(P) > \tilde{D}^2\cdot H_i\}
\\&=: \I\cup \bigcup_{i=1}^r \II_D^{(i)}\cup \bigcup_{i=1}^r \III_D^{(i)},
\end{align*}\noindent
(Note that our notation $\I,\II,\III$ is slightly different from the outline, since we have already gotten rid of ``small'' points.)

In words, what we have done is broken $E(\Z)$ into multiples of rational points (which will be easy to handle)\footnote{In the rank $1$ case all points are multiples of a rational point, so in some sense ``$E(\Z) =: \I$'' would be consistent notation here, but we have not bothered because it would be unnecessarily confusing.}, points of ``medium'' height in their respective cosets, and then points of ``large'' height in their respective cosets. (The curves with points of small height have already been thrown out.) Note that this decomposition is complete because if $P\in E(\Z)$ lies outside the union, then $i(P) =: i > 0$ and $\hat{h}(P)\leq \tilde{D}^2 H_i$, so $\hat{h}(P) < H_i$. Therefore, since $i(P) = i$, we must have $H_i = \tilde{D}^2 H_{i-1}$ by minimality of $P_i$. Thus $\hat{P} < H_{i-1}$. Proceeding inductively, we eventually find that $\hat{P} < (5 - O(\delta))\log{T}$, contradicting $h(P) > (5-\delta)\log{T}$ combined with Lemma \ref{weil vs canonical}.

Let us further write 
\begin{align*}
\III_D^{(i)} &= \bigcup_{\vec{a}\in \{-1,0,1\}^i : a_i > 0} \{P\in \III_D^{(i)}, P\equiv \sum_{j=1}^i a_j P_j\pmod{3}\}\\&=: \bigcup_{\vec{a}\in \{-1,0,1\}^i : a_i > 0} \III_D^{(i,\vec{a})}.
\end{align*}\noindent
In words, we are breaking the points of ``large'' height into their congruence classes modulo $3$. (Since we will be counting points and their negatives together below, we have forced $a_i > 0$ rather than $a_i\neq 0$.)

Given $\vec{a}\in \{-1,0,1\}^i$ with $a_i\neq 0$, we will write $R_{\vec{a}} := \sum_{j=1}^i a_j P_j$. Let us further break $\III_D^{(i,\vec{a})}$ into a set we will show is empty and a set to which we can apply Roth-like techniques. Specifically, write
\begin{align*}
\III_D^{(i,\vec{a})} &= \left\{P\in \III_D^{(i,\vec{a})} : \exists! Q\in E(\Q) : P = 3Q + R_{\vec{a}}; \right.\\&\quad\quad\left.\forall \tilde{R}\in E(\Qbar)\text{ with } 3\tilde{R} = -R_{\vec{a}},\text{ we have } |x(Q) - x(\tilde{R})| > \frac{1}{2}\min_{3\tilde{R}' = -R_{\vec{a}}, \tilde{R}'\neq \tilde{R}} |x(Q) - x(\tilde{R}')|\right\}
\\&\cup \bigcup_{\tilde{R}\in E(\Qbar) : 3\tilde{R} = -R_{\vec{a}}} \left\{P\in \III_D^{(i,\vec{a})} : \exists! Q\in E(\Q) : P = 3Q + R_{\vec{a}};   \right.\\&\quad\quad\quad\quad\quad\quad\quad\quad\quad\left. |x(Q) - x(\tilde{R})|\leq \frac{1}{2}\min_{3\tilde{R}' = -R_{\vec{a}}, \tilde{R}'\neq \tilde{R}} |x(Q) - x(\tilde{R}')|\right\}
\\&=: \IV_D^{i,\vec{a}}\cup \bigcup_{3\tilde{R} = -R_{\vec{a}}} \III_D^{(i,\vec{a},\tilde{R})}.
\end{align*}\noindent
In words, we have written $P\in \III_D^{(i,\vec{a})}$ as $P = 3Q + R_{\vec{a}}$, and split the points up based on the element of the nine-element set $-\frac{1}{3}R$ that $Q$ is close to. $\IV_D^{(i,\vec{a})}$ is the set of points with $Q$ not close to any point in $-\frac{1}{3}R$, which will be empty once $D$ is sufficiently large. (This is because $x(P)$ is large, so $P$ is close to the origin, so that $Q$ is close to such a solution.)

\subsection{$\I$ is small: multiples of rational points are rarely integral}

Let us now begin bounding the sizes of each of the sets $\I,\ldots,\IV$. The sets $\I$ and $\II_D$ require almost no work. The following lemma expresses the fact that rational points rarely have integral multiples: in the rank one case, at worst one has the generator and its negative as integral points (via the theory of lower bounds on linear forms in elliptic logarithms), and in the higher rank case no triple of a rational point is integral on a curve in our family.
\begin{lem}\label{multiples lemma}
Let $E\in \mathcal{F}_*$. Then: $\#|E(\Z)|\leq 2$ when $r=1$, and $I = \emptyset$ otherwise.
\end{lem}\noindent
Before we prove this lemma, we will prove a preparatory lemma on the coefficients of the division polynomials of $E$. Recall that the denominator of the multiplication-by-$n$ map, $\psi_n(P)^2$, is homogeneous in $x,A,B$ of degree $n^2 - 1$ with the usual grading. Write $$\psi_n(P)^2 =: \sum_{\vec{f}\in \N^3 : f_x + 2f_A + 3f_B = n^2-1} c_{\vec{f}}\cdot x^{f_x} A^{f_A} B^{f_B},$$ with $c_{\vec{f}}\in \Z$. The claim is that these $c_{\vec{f}}$ do not grow too fast as $f_x$ decreases. More precisely,
\begin{lem}\label{coefficient bound}
$$c_{\vec{f}}\ll n^{O(1)} O(1)^{(\log{n})^2\cdot (n^2 - 1 - f_x)}.$$
\end{lem}\noindent
It is a theorem of Lang that $c_{\vec{f}}\ll O(1)^{n^2}$ in general (which is only weaker for $f_x\geq (1-o(1))n^2$), but this is not enough for our purposes.
\begin{proof}[Proof of Lemma \ref{coefficient bound}.]
Write $$\psi_n(P) =: y^{1-n\bmod{2}} \sum_{\vec{f}\in \N^3: f_x + 2f_A + 3f_B = 2\floor{\frac{n^2-1}{4}}} C_{\vec{f}}\cdot x^{f_x} A^{f_A} B^{f_B}.$$ We will show that $$C_{\vec{f}}\ll n^{O(1)} O(1)^{(\log{n})^2\cdot \left(2\floor{\frac{n^2-1}{4}} - f_x\right)},$$ from which the bound for $c_{\vec{f}}$ follows. That is, we will show that there are absolute constants $K_1, K_2$, and $K_3$ such that, for all $\vec{f}$, \begin{align}C_{\vec{f}}\leq K_1 n^{K_2} K_3^{(\log{n})^2\cdot \left(2\floor{\frac{n^2-1}{4}} - f_x\right)}.\label{bound on C}\end{align}

First choose $K_1 > 1$ so large that for $n\leq 10^{10}$ the bound $|C_{\vec{f}}|\leq K_1$ holds. Take $K_2 = 1$. Take $K_3$ so large that $10^{10} K_1^3 n^{4K_2 + 10} < K_3^{\log{(1.9)}\cdot \log{n}}$ for all $n > 10^{10}$. The bound will then follow by induction. Specifically, recall the recursive formulas for the division polynomials: for odd indices, $$\psi_{2m+1} = \psi_{m+2}\psi_m^3 - \psi_{m-1}\psi_{m+1}^3,$$ and, for even indices, $$\psi_{2m} = \left(\frac{\psi_m}{2y}\right)\left(\psi_{m+2}\psi_{m-1}^2 - \psi_{m-2}\psi_{m+1}^2\right).$$ So suppose we have proved \eqref{bound on C} for all $n' < n$. From the recursions and induction it follows immediately that the leading coefficient of $\psi_n$ is $n$, which satisfies the claimed bound since $K_1 > 1$, $K_2 = 1$. Hence we may assume $$f_x < 2\floor{\frac{n^2-1}{4}}.$$ For $n$ of the form $n =: 4m+1$, using the recursive formula, we find that $$\psi_{4m+1} = -\psi_{2m-1}\psi_{2m+1}^3 + \left(\frac{\psi_{2m+2}}{y}\right)\left(\frac{\psi_{2m}}{y}\right)^3(x^3 + Ax + B)^2.$$ Expanding and applying the induction hypothesis, we find that the coefficient of $x^{f_x} A^{f_A} B^{f_B}$ in $\psi_{4m+1}$ is, in absolute value, at most a sum of at most $n^6$ terms (corresponding to decompositions $\vec{f} = \vec{e}_1 + \cdots + \vec{e}_4$), each at most $$100 K_1^4 n^{4K_2} K_3^{\log{(2m+2)}^2(8m^2 + 4m - f_x)}.$$ But $\log{(2m+2)}\leq \log{n} - \log{(1.9)}$, so that $$\log{(2m+2)}^2\leq (\log{n})^2 - \log{(1.9)}\cdot \log{n}.$$ Inserting this into the inequality and using $f_x < 2\floor{\frac{n^2-1}{4}}$, we find that $$|C_{\vec{f}}|\leq K_1 n^{K_2} K_3^{(\log{n})^2 \left(2\floor{\frac{n^2-1}{4}} - f_x\right)} \left[100 K_1^3 n^{3K_2 + 6} K_3^{-\log{(1.9)}\cdot \log{n}}\right],$$ and the factor in brackets is smaller than $1$ by hypothesis. For $n$ not congruent to $1$ mod $4$ the argument is exactly the same, using the other recursive relation when $n$ is even.
\end{proof}

This finishes our preparations. Let us now prove Lemma \ref{multiples lemma}.
\begin{proof}[Proof of Lemma \ref{multiples lemma}.]
For the first bound, note that if $nP$ is integral for some $n\geq 1$, then $P$ must be integral. To see this, write $P = \left(\frac{x}{d^2}, \frac{y}{d^3}\right)$ in lowest terms and suppose $d > 1$. Then since $$x(nP) = \frac{x\psi_n(P)^2 - \psi_{n+1}(P)\psi_{n-1}(P)}{\psi_n(P)^2}$$ is the quotient of two homogeneous polynomials of degree $n^2$ and $n^2 - 1$, respectively (again, $x,y,A,B$ are given degrees $1,\frac{3}{2},2$, and $3$, respectively) with the numerator having leading term $x^{n^2}$, we see that, on clearing denominators, $$x(nP) = \frac{x^{n^2} + (\in d\Z)}{(\in d\Z)},$$ which is not an integer since $(x,d) = 1$ by hypothesis.

So if $P = P_1$ is not integral we are done for the rank $1$ case. If it is integral, then the claim is that none of its multiples $nP$, $n > 1$, are also integral. Indeed, since $P$ is integral, we find that $h(P) > (5-\delta)\log{T}$ since $E\in \mathcal{F}_*$. 

Let us first show that $nP$ is not integral for $1 < n \ll O(1)^{\sqrt{\log{T}}}$. Of course it suffices to show that the denominator $d_n^2$ in lowest terms of $x(nP)$ is larger than $1$ for these $n$. But Lemma 29 of \cite{stange} (or, equivalently, Proposition 4.2.3 in \cite{mahe}) allows us to do this. Indeed, we find that $$\log{(d_n^2)}\geq \log{(\psi_n(P)^2)} - \frac{n^2}{4}\log{|\Delta|}\geq \log{(\psi_n(P)^2)} - \frac{3n^2}{2}\log{T} - O(1).$$

By Lemma \ref{coefficient bound}, the coefficient of $x^k$ is at most $$\ll n^{O(1)} \left(O(1)^{(\log{n})^2} T\right)^{n^2 - 1 - k}.$$ Hence since $|x(P)|\geq T^{5-\delta}$ is \emph{much} larger than $T$, we find that $\psi_n(P)^2$ is dominated by its top term. Specifically, for $n \ll O(1)^{\sqrt{\log{T}}}$, $$\psi_n(P)^2\geq |x(P)|^{n^2-1}\left(1 - n^{O(1)} O(1)^{(\log{n})^2} T^{-\Omega(1)}\right)\gg |x(P)|^{n^2-1},$$ so that $$\log{(d_n^2)}\geq (n^2-1)h(P) - \frac{3n^2}{2}\log{T} - O(1)\geq \left(9 - O(\delta)\right)\log{T} - O(1),$$ which is positive. Thus $d_n > 1$ and so $x(nP)$ is not integral for $n\ll O(1)^{\sqrt{\log{T}}}$. This in fact completes the first estimate since it shows that no integral point is thrice a rational point in general as well (for this application we could simply use Lang's coefficient bound, of course).

Thus it remains to show that $nP$ is not integral for $n\gg O(1)^{\sqrt{\log{T}}}$. This will follow from David's bounds on linear forms in elliptic logarithms --- in fact we will show that $nP$ is not integral for $n\gg \log{T}\sqrt{\log\log{T}}\log\log\log{T}$. To do this we apply the Corollary of equation (26) in \cite{gebelpethozimmer}. Let us translate their notation into ours. Recall that, for us, $r=1$, so that their $C\ll 1$. Moreover, since our curves have no torsion, their $g=1$. Their $N$ is our $n$. Their $\mu_\infty = \log\max(|A|^{\frac{1}{2}}, |B|^{\frac{1}{3}})\leq \log{T} + O(1)$.

They define the real period $\omega_1$ to be $$\omega_1 := 2\int_\rho^\infty \frac{dx}{\sqrt{x^3 + Ax + B}},$$ where $\rho\in \R$ is the largest real solution of $\rho^3 + A\rho + B = 0$. Let us show that $$T^{-\frac{1}{2}}\ll \omega_1\ll T^{-\frac{1}{2} + O(\delta)}.$$ Let $\rho',\rho''\in \C$ be the other two roots. Since $A$ and $B$ satisfy $$T^{1-O(\delta)}\ll |A|^{\frac{1}{2}}, |B|^{\frac{1}{3}}\ll T,$$ it follows by the reverse triangle inequality that the same bounds hold for $|\rho|, |\rho'|$, and $|\rho''|$. Now the integral over $[10^{10}T, \infty)$ is $\asymp T^{-\frac{1}{2}}$ since $x^3 + Ax + B\gg |x|^3$ there. Hence, since the integrand is positive, the lower bound on $\omega_1$ follows. For the upper bound, we split into cases. If $\rho',\rho''$ are not real, then $\Re(\rho') = \Re(\rho'') = -\frac{\rho}{2}$ and $\Im(\rho') = -\Im(\rho'') = \frac{\rho' - \rho''}{2}$. In this case, on $(\rho, 10^{10}T)$ $$x^3 + Ax + B\gg (x-\rho) |\rho' - \rho''|^2.$$ If $\rho',\rho''$ are real, then on $(\rho, 10^{10}T)$ $$x^3 + Ax + B = (x-\rho)(x-\rho')(x-\rho'')\geq (x-\rho)(\rho - \rho')(\rho - \rho'').$$ Since the discriminant of $x^3 + Ax + B$ is $\gg T^{6 - O(\delta)}$, applying Mahler's bound on the bottom of page 261 in \cite{mahler}, in both cases it follows that $$x^3 + Ax + B\gg (x-\rho)\cdot T^{2 - O(\delta)}$$ on the interval. Hence the integral over the interval is $$\int_\rho^{10^{10}T} \frac{dx}{\sqrt{x^3 + Ax + B}}\ll T^{-1 + O(\delta)}\int_\rho^{10^{10}T} \frac{dx}{\sqrt{x-\rho}}\ll T^{-\frac{1}{2} + O(\delta)},$$ completing the argument.

It follows that their $c_1'\gg T^{\frac{1}{2} - O(\delta)}$. Note also that their $h\ll \log{T}$. The bound $|\rho|, |\rho'|, |\rho''|\ll T$ implies that their $\xi_0\ll T$. Finally, we turn to the expression $\frac{3\pi |u_1|^2}{\omega_1^2 \Im(\tau)}$ defining their $\log{V_1}$. Since we may take $\tau$ in the classical fundamental domain for $\SL_2(\Z)$ acting on the upper half plane, we have $\Im(\tau)\gg 1$. Now, $u_1$, the elliptic logarithm of our $P = P_1 =: (\xi,\eta)$, satisfies $$u_1 = \frac{1}{\omega_1}\int_\xi^\infty \frac{dx}{\sqrt{x^3 + Ax + B}}\ll \xi^{-\frac{1}{2}} T^{\frac{1}{2} + O(\delta)}.$$ Thus $$\frac{3\pi |u_1|^2}{\omega_1^2 \Im(\tau)}\ll \frac{|u_1|^2}{\omega_1^2}\ll |x(P)|^{-1} T^{2 + O(\delta)}.$$ But $|x(P)|\gg T^{5-\delta}$ implies that this is $$\ll T^{-3 + O(\delta)}.$$ Therefore their $\log{V_1}$ satisfies $$\log{V_1}\ll \hat{h}(P_1).$$ Finally, their $\lambda_1 = \hat{h}(P_1)$ in the rank one case.

This completes the translation of their notation. Their Corollary now reads (since certainly any integral point $P'$ satisfies the hypothesis of their Proposition, which is $x(P')\gg T$ --- $x(P')$ is positive since $x(P')^3 + A x(P') + B$ is):
\begin{cor}[Cf.\ equation (26) of \cite{gebelpethozimmer}.]
For $E\in \mathcal{F}_*$ of rank one and generator $P = P_1$, if $nP$ is integral and $n\gg 1$, then $$n^2\ll (\log{T})^2 \log{n} (\log\log{n})^2.$$
\end{cor}

It follows that, if $nP$ is integral, then $n\ll \log{T}\sqrt{\log\log{T}}\log\log\log{T}$. Since we have already shown that if $n > 1$ then $n\gg O(1)^{\sqrt{\log{T}}}$, this completes the argument.
\end{proof}

Note that we have now completely handled the cases of $\rank(E) = 0$ or $1$. Hence from now on we may assume $\rank(E)\geq 2$.

\subsection{$\II$ is small: integral points repel in the Mordell-Weil lattice}

Let $1 < J < 2$ be a parameter which we will choose at the end ($J$ will depend on $r$ for $r\ll 1$). Write $J =: 2\cos{\theta}$. We encode the fact that integral points repel in the Mordell-Weil lattice with the following lemma.
\begin{lem}
$$\#|\II_D^{(i)}|\leq 2\left\lceil\frac{\log{\tilde{D}}}{\log{J}}\right\rceil \max_{S\subseteq \RP^{r-1} : \forall v\neq w\in S, |\langle v, w\rangle|\leq \cos{\theta} + O(\delta)} \#|S|.$$
\end{lem}
We will bound the maximum occurring in this bound with a bound on codes in $\RP^n$ via linear programming techniques for $n\ll 1$ and a simpleminded volume estimate for $n\gg 1$.
\begin{proof}
It suffices to prove that the number of points with height in an interval $[m,M]$ is $$\leq 2\left\lceil\frac{\log{\left(\frac{M}{m}\right)}}{2\log{J}}\right\rceil \max_{S\subseteq \RP^{r-1} : \forall v\neq w\in S, |\langle v, w\rangle|\leq \cos{\theta} + O(\delta)} \#|S|.$$ To see this, note that $$[m,M]\subseteq \bigcup_{i=1}^{\left\lceil\frac{\log{\left(\frac{M}{m}\right)}}{2\log{J}}\right\rceil} [m (J^2)^i, m (J^2)^{i+1}],$$ so that it suffices to prove this bound for an interval of the form $[m, J^2m]$. But now if $h(R)\leq h(P)\leq J^2 h(R)$, then by Lemma \ref{helfgott-mumford gap principle}, $$\cos{\theta_{P,R}}\leq \frac{J}{2} + O(\delta) = \cos{\theta} + O(\delta).$$ Therefore the map $\{P\in E(\Z) : h(P)\in [m, J^2 m]\}/\pm\to \RP^{r-1}$ via $\pm P\mapsto \{\pm P\otimes \frac{1}{\sqrt{\hat{h}(P)}}\}$ (the projection to $\RP^{r-1}$ of the nonzero point $P\in \R^r\iso E(\Q)\otimes_\Z \R$) is injective (since $\cos{\theta_{P,R}} < 1$ if $P\neq R$ once $\delta\ll_J 1$). Moreover the image satisfies the condition that for every $v\neq w$ in the image, $|\langle v, w\rangle| = \cos{\theta_{v,w}}\leq \cos{\theta} + O(\delta)$, as desired. This completes the proof of the second bound.
\end{proof}

\subsection{$\III$ is small and $\IV$ is empty: an explicit bivariate Roth's Lemma}

For $\III_D^{(i,\vec{a},\tilde{R})}$ and $\IV_D^{(i,\vec{a})}$ we will follow Siegel's proof of his finiteness theorem. Write $C := (5-\delta)\tilde{D}^2$, so that for every $P\in \III_D^{(i,\vec{a})}$ we have $h(P) > C\log{T}$. Note also that
\begin{lem}\label{height of R is small compared to P}
Let $P\in \III_D^{(i,\vec{a})}$. Then: $$h(R_{\vec{a}}), \hat{h}(R_{\vec{a}})\leq \left(\frac{1}{D^2} + O(\delta)\right) h(P).$$
\end{lem}
\begin{proof}
Observe that
\begin{align*}
\hat{h}\left(\sum_{j=1}^i a_j P_j\right)&\leq \sum_{j=1}^i (i-j+1)\hat{h}(P_i)
\\&\leq \sum_{j=1}^i \frac{i-j+1}{\tilde{D}^{2(i-j+1)}} \hat{h}(P)
\\&= \left(\sum_{j=1}^i j \tilde{D}^{-2j}\right)(1 + O(\delta)) h(P),
\end{align*}
where the first step follows from \eqref{bound on Rs} and the second follows from the definition of $\III_D^{(i)}$, plus Lemma \ref{weil vs canonical}. But in general $$\sum_{\ell=1}^k \ell x^{-\ell}\leq \frac{x}{(x-1)^2},$$ so that we find that $$h(R_{\vec{a}})\leq \left(\frac{1}{D^2} + O(\delta)\right) h(P)$$ by \eqref{def of tilde D} and Lemma \ref{weil vs canonical}, as desired.
\end{proof}

Having established this estimate, let us now prove:
\begin{lem}\label{the famous previous lemma}
Let $P\in \III_D^{(i,\vec{a})}$. Then:
\begin{align*}
\frac{\log{\prod_{3\tilde{R} = -R_{\vec{a}}} \left|x(Q) - x(\tilde{R})\right|^{-1}}}{h(P)}\geq \frac{1}{2} - \max\left(\frac{19\log{T}}{h(P)}, \frac{19}{D^2}\right) - O(\delta),
\end{align*}
and
\begin{align*}
\left(1 + D^{-1} + O(\delta)\right)^{-2}\leq \frac{h(P)}{9h(Q)}\leq\left(1 - D^{-1} - O(\delta)\right)^{-2}.
\end{align*}
\end{lem}

\begin{proof}
Observe that
\begin{align}
\frac{1}{2} &= \frac{\log{|x(P)|^{\frac{1}{2}}}}{h(P)}\nonumber\\&= \frac{\log{|x(3Q + R_{\vec{a}})|^{\frac{1}{2}}}}{h(3Q + R_{\vec{a}})}.\label{expression for one half}
\end{align}\noindent
Let us examine the numerator and denominator of this expression.

First, the denominator. Note that
\begin{align*}
\sqrt{h(P)} &= \sqrt{h(3Q + R_{\vec{a}})}\\&\geq \sqrt{\hat{h}(3Q + R_{\vec{a}}) - O(\delta\log{T})}
\\&= \sqrt{\hat{h}(3Q+R_{\vec{a}})}\left(1 - O(\delta)\right)
\\&\geq \left(3\sqrt{\hat{h}(Q)} - \sqrt{\hat{h}(R_{\vec{a}})}\right)\left(1 - O(\delta)\right)
\\&\geq \left(3\sqrt{h(Q)} - \frac{\sqrt{h(P)}}{D}\right)\left(1 - O(\delta)\right),
\end{align*}
where we have used the triangle inequality for $\sqrt{\hat{h}}$ and Lemma \ref{height of R is small compared to P}.

Therefore
\begin{align*}
\sqrt{h(P)}&\geq 3\sqrt{h(Q)}\left(1 + D^{-1} + O(\delta)\right)^{-1}.
\end{align*}
The same argument works to prove that
\begin{align*}
\sqrt{h(Q)}\geq \frac{1}{3}\sqrt{h(P)}\left(1 - D^{-1} - O(\delta)\right)
\end{align*}
as well. This proves the second statement of the Lemma.

Now we move to the numerator in \eqref{expression for one half}. Observe that $$x(3Q) = \frac{x(Q) \psi_3(Q)^2 - \psi_2(Q)\psi_4(Q)}{\psi_3(Q)^2}.$$ Note also that
\begin{align*}
9\prod_{3\tilde{R} = -R_{\vec{a}}} (x(Q) - x(\tilde{R})) &= \psi_3(Q)^2\left(x(Q) - \frac{\psi_2(Q)\psi_4(Q)}{\psi_3(Q)^2} - x(R_{\vec{a}})\right) \\&= \psi_3(Q)^2\left(x(3Q) - x(R_{\vec{a}})\right),
\end{align*} since both are polynomials in $x(Q)$ of degree $9$ with leading coefficient $9$ and roots exactly at $x(Q) = x(\tilde{R})$ for some $\tilde{R}$ with $3\tilde{R} = -R_{\vec{a}}$. But then, since in general $$x(W + Z) = \frac{(y(W) - y(Z))^2 - (x(W) + x(Z))(x(W) - x(Z))^2}{(x(W) - x(Z))^2},$$ we have that
\begin{align*}
&x(3Q+R_{\vec{a}})\cdot \left(9\prod_{3\tilde{R} = -R_{\vec{a}}} (x(Q) - x(\tilde{R}))\right)^2 \\&\quad\quad= \psi_3(Q)^4 \left((y(3Q) - y(R_{\vec{a}}))^2 - (x(3Q) + x(R_{\vec{a}}))(x(3Q) - x(R_{\vec{a}}))^2\right).
\end{align*}

Now from the equation $y^2 = x^3 + Ax + B$, we find that $y\ll (|x| + T)^{\frac{3}{2}}$ in general. Also,
\begin{align*}
|x(R_{\vec{a}})|&\leq \exp(h(R_{\vec{a}}))\\&\leq \exp\left(\frac{h(P)}{D^2}(1 + O(\delta))\right)\\&= |x(P)|^{\frac{1}{D^2} + O(\delta)}.
\end{align*}

Therefore, as we saw in the proof of Lemma \ref{helfgott-mumford gap principle}, by writing out the numerator and denominator, $|x(3Q)| = |x(P - R_{\vec{a}})|\ll |x(R_{\vec{a}})|$ since $x(R_{\vec{a}})$ is much smaller than $x(P)$ in absolute value. But if $|x(Q)|\geq 10^{10}T$, then $|x(3Q)|\gg |x(Q)|$ since it is a quotient of two polynomials dominated by their leading terms. Therefore we find that in general $|x(Q)|, |x(3Q)|\ll |x(R_{\vec{a}})| + T$ and so  $|y(Q)|, |y(3Q)|\ll (|x(R_{\vec{a}})| + T)^{\frac{3}{2}}$.

Therefore
\begin{align*}
&\left|x(3Q+R_{\vec{a}})\cdot \left(9\prod_{3\tilde{R} = -R_{\vec{a}}} (x(Q) - x(\tilde{R}))\right)^2\right| \\&\quad\quad= \left|\psi_3(Q)^4 \left((y(3Q) - y(R_{\vec{a}}))^2 - (x(3Q) + x(R_{\vec{a}}))(x(3Q) - x(R_{\vec{a}}))^2\right)\right|\\&\quad\quad\ll (|x(R_{\vec{a}})| + T)^{19}\\&\quad\quad\ll \max\left(T^{19}, |x(P)|^{\frac{19}{D^2} + O(\delta)}\right).
\end{align*}

Written another way,
\begin{align*}
\log{|x(3Q + R_{\vec{a}})|^{\frac{1}{2}}}\leq \log{\prod_{3\tilde{R} = -R_{\vec{a}}} \left|x(Q) - x(\tilde{R})\right|^{-1}} + \max\left(19\log{T}, \frac{19h(P)}{D^2} + O(\delta)\right) + O(1).
\end{align*}

Therefore, returning to \eqref{expression for one half}, we find that:
\begin{align*}
\frac{1}{2} &\leq \frac{\log{\prod_{3\tilde{R} = -R_{\vec{a}}} \left|x(Q) - x(\tilde{R})\right|^{-1}}}{h(P)} + \max\left(\frac{19\log{T}}{h(P)}, \frac{19}{D^2}\right) + O(\delta).
\end{align*}\noindent
This completes the proof.
\end{proof}

Let us now show that, once $D$ is suitably chosen, $\IV_D^{(i,\vec{a})}$ is empty. (Recall that $C = (5-\delta)\tilde{D}^2$.)
\begin{lem}
Suppose $$\frac{576}{C} + \frac{72}{D^2} + \max\left(\frac{19}{C}, \frac{19}{D^2}\right) < \frac{1}{2}.$$ Then $\IV_D^{(i,\vec{a})} = \emptyset$.
\end{lem}
\begin{proof}
Suppose $P\in \IV_D^{(i,\vec{a})}$. Then, by definition, $$\prod_{3\tilde{R} = -R_{\vec{a}}} \left|x(Q) - x(\tilde{R})\right|^{-1}\ll \min_{\tilde{R}\neq \tilde{R}', 3\tilde{R} = 3\tilde{R}' = -R_{\vec{a}}} |x(\tilde{R}) - x(\tilde{R}')|^{-9}.$$ Now, as we saw in the previous lemma, as polynomials in $x(Q)$, $$9\prod_{3\tilde{R} = -R_{\vec{a}}} (x(Q) - x(\tilde{R})) = \psi_3(Q)^2 x(Q) - \psi_2(Q)\psi_4(Q) - \psi_3(Q)^2 x(R_{\vec{a}}).$$ This is homogeneous of degree $9$ in $x(Q), x(R), A, B$ when the variables are given degrees $1, 1, 2, 3$, respectively. Therefore the coefficients of $x(Q)$ in the first two terms (namely, $\psi_3(Q)^2 x(Q) - \psi_2(Q)\psi_4(Q)$) are bounded in absolute value by $\ll T^8$. Thus the polynomial has na\"{\i}ve height, in the sense of Bugeaud and Mignotte \cite{bugeaudmignotte}, at most $8\log{T} + h(R_{\vec{a}})$. To see this, clear the denominator of $x(R_{\vec{a}})$ so that the polynomial is an integral polynomial and then the estimate is clear. Therefore by the estimate on page 262 of Mahler \cite{mahler}, we find that $$\min_{\tilde{R}\neq \tilde{R}', 3\tilde{R} = 3\tilde{R}' = -R_{\vec{a}}} |x(\tilde{R}) - x(\tilde{R}')|\gg T^{-64} H(R_{\vec{a}})^{-8},$$ and hence that $$\prod_{3\tilde{R} = -R_{\vec{a}}} \log{\left|x(Q) - x(\tilde{R})\right|^{-1}}\leq 576\log{T} + 72h(R_{\vec{a}}) + O(1).$$

Therefore by Lemma \ref{the famous previous lemma} it follows that $$\frac{1}{2} \leq \frac{576\log{T} + 72h(R_{\vec{a}}) + O(1)}{h(P)} + \max\left(\frac{19\log{T}}{h(P)}, \frac{19}{D^2}\right) + O(\delta).$$ Applying Lemma \ref{height of R is small compared to P}, we see that $$\frac{1}{2} \leq \frac{576\log{T}}{h(P)} + \frac{72}{D^2} + \max\left(\frac{19\log{T}}{h(P)}, \frac{19}{D^2}\right) + O(\delta).$$

The desired contradiction now follows (once $\delta\ll 1$) by using the inequality $h(P) > C\log{T}$.
\end{proof}

Finally we will bound the size of $\III_D^{(i,\vec{a},\tilde{R})}$ for $D$ suitably chosen. The idea here is that, roughly, we have obtained the inequality $$\frac{\log{\prod_{3\tilde{R} = -R_{\vec{a}}} \left|x(Q) - x(\tilde{R})\right|^{-1}}}{h(Q)}\geq 4.49,$$ and now $Q$ is very close to some $\tilde{R}$. Therefore, roughly, this tells us that $|x(Q) - x(\tilde{R})|\leq H(Q)^{-4.48}$, and so $x(Q)$ is a Roth-type rational approximation to $x(\tilde{R})$. But Roth's theorem requires many such rational approximations to reach a contradiction, and hence provides a poor bound on their number for our purposes. In fact $x(Q)$ is \emph{also} a Siegel-type rational approximation, in the sense that $x(\tilde{R})$ is of degree $9$ over $\Q$, and $\sqrt{2\deg{x(\tilde{R})}} = \sqrt{18} = 4.24... < 4.48$. Moreover $x(Q)$ has very large height compared to $x(\tilde{R})$, so if we are very careful with how we prove Siegel's theorem on Diophantine approximation (namely, via Roth's lemma for bivariate polynomials), we will be able to conclude.

So let $c < 1$ be another parameter (which we will choose such that $1-c\gg 1$). Given $c$ and $D$, we may bound the size of $\III_D^{(i,\vec{a},\tilde{R})}$ as follows.
\begin{lem}
Suppose $s\in \Z^+$ is such that $$\left(\frac{\sqrt{2}c}{3} - \frac{1}{(\kappa - 1)^s}\right)\kappa - \frac{1 + \frac{1}{(\kappa - 1)^s}}{(D-1)^2}\left(9 + \frac{\kappa + 1}{(c^{-2} - 1)}\right) > 2.$$ Then $\#|\III_D^{(i,\vec{a},\tilde{R})}|\leq 2s$.
\end{lem}
\begin{proof}
Let $P\in \III_D^{(i,\vec{a},\tilde{R})}$. Note that, for all $\tilde{R}'\neq \tilde{R}$ such that $3\tilde{R}' = -R_{\vec{a}}$, $$|x(Q) - x(\tilde{R}')| > \frac{1}{2}|x(\tilde{R}) - x(\tilde{R}')|$$ by the triangle inequality. Therefore $$\prod_{3\tilde{R}' = -R_{\vec{a}}, \tilde{R}'\neq \tilde{R}} \left|x(Q) - x(\tilde{R}')\right|\gg \prod_{3\tilde{R}' = -R_{\vec{a}}, \tilde{R}'\neq \tilde{R}} |x(\tilde{R}) - x(\tilde{R}')|.$$ By a bound of Mahler (the last line on page 262 of \cite{mahler}), $$\prod_{3\tilde{R}' = -R_{\vec{a}}, \tilde{R}'\neq \tilde{R}} |x(\tilde{R}) - x(\tilde{R}')|\gg T^{-56}H(R_{\vec{a}})^{-7}.$$ Hence, by Lemma \ref{the famous previous lemma},
\begin{align*}
\frac{\log{\left|x(Q) - x(\tilde{R})\right|^{-1}}}{h(P)}\geq \frac{1}{2} - \max\left(\frac{19}{C}, \frac{19}{D^2}\right) - \frac{56}{C} - \frac{7}{D^2} + O(\delta).
\end{align*}
Next, applying the second part of Lemma \ref{the famous previous lemma}, we therefore find that
\begin{align}\label{end of calculation of kappa}
\frac{\log{\left|x(Q) - x(\tilde{R})\right|^{-1}}}{h(Q)}\geq \left(\frac{9}{2} - \max\left(\frac{171}{C}, \frac{171}{D^2}\right) - \frac{504}{C} - \frac{63}{D^2} + O(\delta)\right)\left(1 + D^{-1} + O(\delta)\right)^{-2}.
\end{align}\noindent
Write $$\kappa := \left(\frac{9}{2} - \max\left(\frac{171}{C}, \frac{171}{D^2}\right) - \frac{504}{C} - \frac{63}{D^2}\right)\left(1 + D^{-1}\right)^{-2} + O(\delta).$$ Then $$|x(Q) - x(\tilde{R})|\leq H(Q)^{-\kappa}.$$ That is, $x(Q)\in \Q$ is a rational approximation to $x(\tilde{R})\in \Qbar$ with exponent $\kappa$. Moreover,
\begin{align}
h(Q)&\geq \frac{1}{9}h(P)\left(1 - D^{-1} - O(\delta)\right)^2\nonumber
\\&\geq \frac{(D-1-O(\delta))^2}{9}\hat{h}(R_{\vec{a}})\nonumber
\\&\geq (D-1-O(\delta))^2 \hat{h}(\tilde{R})\nonumber
\\&\geq (D-1-O(\delta))^2 h(\tilde{R})\label{height of tilde R is small compared to Q}
\end{align}\noindent
so that $x(Q)$ is a ``large'' rational approximation of $x(\tilde{R})$ as well. To bound the number of these, we will run through the usual argument for Siegel's theorem on Diophantine approximation via Roth's lemma, except we will be explicit and careful in our bounds.

Write $\alpha := x(\tilde{R})$ (whence $\deg{\alpha}\leq 9$ and $|\alpha|\ll T + |x(R_{\vec{a}})|$) and let us suppose there were $s+1$ such approximations --- i.e.\ $\lambda_i\neq \lambda_j$ satisfying:
\begin{enumerate}
\item $\lambda_i\in \Q$,
\item $|\lambda_i|\ll T + |x(R_{\vec{a}})|$,
\item $|\lambda_i - \alpha|\leq H(\lambda_i)^{-\kappa}$,
\item $h(\lambda_i)\geq (D-1-O(\delta))^2 h(\alpha)$,
\item $h(\lambda_i)\geq \frac{C}{9}(1-D^{-1}-O(\delta))^2 \log{T}$.
\end{enumerate}\noindent
Let us also suppose, without loss of generality, that $H(\lambda_{s+1})\geq H(\lambda_{s-1})\geq\cdots\geq H(\lambda_1)$.

Note that, by rationality of the $\beta_i$ we have that $$\frac{1}{H(\lambda_{i-1})H(\lambda_i)}\leq |\lambda_{i-1} - \lambda_i|\leq 2H(\lambda_{i-1})^{-\kappa}.$$ Hence $$H(\lambda_i)\geq \frac{1}{2}H(\lambda_{i-1})^{\kappa-1}$$ --- i.e., $$h(\lambda_i)\geq (\kappa - 1) h(\lambda_{i-1}) + O(1).$$ Therefore $$h(\lambda_{s+1})\geq (\kappa - 1)^s h(\lambda_1) + O(s).$$ Hence $\lambda_{s+1}$ and $\lambda_1$ are very far apart in height, and it is these rational approximations that we will use. We will write $\beta_2 := \lambda_{s+1}$ and $\beta_1 := \lambda_1$.

Now let $d_1 > d_2\in \Z^+$ be such that $$\left|\frac{d_2}{d_1} - \frac{h(\beta_1)}{h(\beta_2)}\right|\leq \frac{1}{d_1^2}.\footnote{Of course infinitely many such $d_1$ and $d_2$ exist if $\frac{h(\beta_1)}{h(\beta_2)}$ is irrational, but since we do not require $(d_1, d_2) = 1$, such $d_1, d_2$ exist in case the ratio of heights is rational as well!}$$ We will take $d_1, d_2\to\infty$ at the end of the argument, so any error terms suppressed by factors of $d_1$ or $d_2$ will be negligible.

Let $t := \frac{c\sqrt{2}}{3}$ and let
\begin{align*}
\deg{\alpha}\cdot K &:= \deg{\alpha}\cdot d_1d_2\cdot \frac{t^2}{2}\cdot \left(1 + t^{-1}(d_1^{-1} + d_2^{-1})\right)^2
\\&= d_1d_2 c^2\left(1 + \frac{3}{cd_1\sqrt{2}} + \frac{3}{cd_2\sqrt{2}}\right)^2
\\&\leq d_1d_2 (c + O(\delta))^2
\end{align*}\noindent
once $d_1, d_2\gg_{c,\delta} 1$. An application of Siegel's lemma gives us the following:
\begin{claim}
There is a nonzero $p\in \Z[x,y]$ such that $$(\d_x^k\d_y^\ell p)(\alpha, \alpha) = 0$$ for all nonnegative integers $k,\ell$ with $$\frac{k}{d_1} + \frac{\ell}{d_2}\leq t,$$ and such that $$H(p)\leq O(H(\tilde{R}))^{\frac{d_1 + d_2}{c^{-2} - 1 - O(\delta)}}.$$
\end{claim}
\begin{proof}[Proof of Claim.]
We apply Siegel's lemma in the form of Bombieri-Gubler Lemma 2.9.1 \cite{bombierigubler}. Indeed, we are imposing the conditions $\sum_{0\leq i\leq d_1, 0\leq j\leq d_2} a_{ij}\alpha^{i+j-k-\ell}{d_1\choose k}{d_2\choose k} = 0$ on the coefficients $a_{ij}\in \Z$ of $P$. But recall that we have the relation $$\den\cdot \alpha^{\deg{\alpha}} = f(\alpha)$$ with $$f(z) := \den\cdot z^{\deg{\alpha}} - \den\cdot g(z),$$ and $g(z)\in \Q[z]$ the minimal polynomial of $\alpha$ (here $\den$ is the least positive integer such that $\den\cdot g\in \Z[z]$). Multiplying our relations through by $\den^{d_1 + d_2 - \deg{\alpha} + 1}$ and repeatedly applying this relation reduces us to forcing $\deg{\alpha}$ times as many conditions (but now with integral coefficients) for each condition with coefficients in $\Q(\alpha)$. Importantly, since the coefficients of $f(z)\in \Z[z]$ are all of absolute value at most $O(H(\alpha))$ and we apply the relation $\leq d_1 + d_2$ times, the resulting linear conditions on $a_{ij}$ have coefficients bounded in absolute value by $$\ll O(1)^{d_1 + d_2} H(\tilde{R})^{d_1 + d_2},$$ where we get an $O(1)^{d_1+d_2}H(\tilde{R})^{d_1+d_2}$ from the $\alpha^{i+j-k-\ell}$ terms, and an $O(1)^{d_1 + d_2}$ from the binomial coefficients and the sum.

To conclude, we note that the number of variables $a_{ij}$ is $(d_1 + 1)(d_2 + 1)$, and the number of equations is $\deg{\alpha}\cdot \#\left|\{\frac{k}{d_1} + \frac{\ell}{d_2}\leq t\}\right|$, which is at most $\deg{\alpha}\cdot K$ by Bombieri-Gubler page 158 \cite{bombierigubler}. Now apply Lemma 2.9.1 of \cite{bombierigubler}.
\end{proof}

So let $p$ be such a polynomial. Following Bombieri-Gubler, we define the index of vanishing of a polynomial $q\in \Z[x,y]$ at a point $(\xi_1, \xi_2)$ to be $$\ind(q, \vec{\xi}) := \min\left\{\frac{k}{d_1} + \frac{\ell}{d_2} : k,\ell\geq 0, (\d_x^k\d_y^\ell q)(\xi_1, \xi_2)\neq 0\right\}.$$ As Bombieri-Gubler note, $\ind(\cdot, \vec{\xi})$ is a non-Archimedean valuation on $\Z[x,y]$, and $$\ind(\d_x^a\d_y^b q, \vec{\xi})\geq \ind(q,\vec{\xi}) - \frac{a}{d_1} - \frac{b}{d_2}.$$

By construction $\ind(p, (\alpha, \alpha))\geq t$. To show that $\ind(p, (\beta_1, \beta_2))$ is small, we will use an improved bivariate form of Roth's lemma. Specifically, we will prove:
\begin{claim}
$$\ind(p, \vec{\beta})\leq \frac{d_2}{d_1} + \frac{(1 + \frac{d_2}{d_1})}{(c^{-2} - 1)(D-1)^2} + O(\delta).$$
\end{claim}\noindent
We will simply follow Bombieri-Gubler and be more careful in the bivariate case.
\begin{proof}[Proof of Claim.]
Write
\begin{align}\label{better expression for U(x)}
U(x) := \det\left(\sum_{0\leq k\leq d_1} {k\choose i} a_{kj} x^{k-i}\right)_{0\leq i,j\leq d_2}.
\end{align}\noindent
Note that
\begin{align}\label{latter expression for U(x)}
U(x) = \det\left(\frac{\d_x^i\d_y^j p}{i!j!}\right)_{0\leq i,j\leq d_2}
\end{align}\noindent
as polynomials in $\Z[x,y]$, since the latter is simply $U(x)$ times $\det\left({j\choose i} y^{j-i}\right)_{0\leq i,j\leq d_2} = 1$. But \eqref{latter expression for U(x)} is proportional to the Wronskian of $p$, whence it does not vanish identically as a polynomial in $x,y$ (i.e., in $x$) by Wronski's theorem.

Now, by expanding out the determinant in \eqref{better expression for U(x)} as a sum over permutations, we find that $$\deg{U}\leq d_1 + (d_1 - 1) + \cdots + (d_1 - d_2) = (d_2 + 1)\left(d_1 - \frac{d_2}{2}\right).$$ Also, by examining the absolute value of the coefficients of $U$ via the same sum over permutations, we find that
\begin{align*}
H(U)&\leq O(1)^{d_1d_2} H(p)^{d_2 + 1}
\\&\leq O(1)^{d_1d_2} H(\tilde{R})^{\frac{(d_1 + d_2)(d_2 + 1)}{c^{-2} - 1 - O(\delta)}}.
\end{align*}

But now for a univariate polynomial $f(x)\in \Z[x]$, $(qx - p)^k\vert f(x)$ (which implies $H(f)\geq H\left(\frac{p}{q}\right)^k$) if $f$ vanishes to order $k$ at $\frac{p}{q}$. Hence $$H(U)\geq H(\beta_1)^{d_1 \ind(W,\vec{\beta}) - 1},$$ or, written another way, $$\ind(W,\vec{\beta})\leq \frac{h(U)}{d_1 h(\beta_1)} + O\left(\frac{1}{d_1}\right).$$

But, applying the fact that $\ind(\cdot, \vec{\beta})$ is a non-Archimedean valuation, $$\ind(W,\vec{\beta}) = \ind(U,\vec{\beta})\geq \min_{\sigma\in S_{d_2+1}} \sum_{a=0}^{d_2} \ind(\d_x^a\d_y^{\sigma(a)} p, \vec{\beta}).$$ But $$\ind(\d_x^a\d_y^{\sigma(a)} p, \vec{\beta})\geq \max\left(\ind(p, \vec{\beta}) - \frac{a}{d_1}, 0\right) - \frac{\sigma(a)}{d_1},$$ so that this sum is simply
\begin{align*}
&-\frac{d_2(d_2 + 1)}{2d_1} + \sum_{0\leq a\leq \min(d_2, d_1\cdot \ind(p, \vec{\beta}))} \ind(p,\vec{\beta}) - \frac{a}{d_1} \\&= -\frac{d_2(d_2 + 1)}{2d_1} + \begin{cases} (d_2 + 1)\left(\ind(p, \vec{\beta}) - \frac{d_2}{2d_1}\right) & \ind(p,\vec{\beta}) > \frac{d_2}{d_1},\\ \left(\floor{d_1\ind(p,\vec{\beta})} + 1\right)\cdot \ind(p, \vec{\beta}) - \frac{\floor{d_1 \ind(p,\vec{\beta})}(\floor{d_1 \ind(p,\vec{\beta})} + 1)}{2d_1} & \ind(p,\vec{\beta})\leq \frac{d_2}{d_1}.\end{cases}
\end{align*}

In the first case we derive the inequality $$\ind(p,\vec{\beta})\leq \frac{d_2}{d_1} + \frac{(1 + \frac{d_2}{d_1})\cdot \frac{h(\tilde{R})}{h(\beta_1)}}{c^{-2} - 1 - O(\delta)} + O(\delta).$$ In the second case we start with the inequality $\ind(p,\vec{\beta})\leq \frac{d_2}{d_1}$ anyway.

Therefore $$\ind(p,\vec{\beta})\leq \frac{d_2}{d_1} + \frac{(1 + \frac{d_2}{d_1})\cdot \frac{h(\tilde{R})}{h(\beta_1)}}{c^{-2} - 1 - O(\delta)} + O(\delta).$$ Recall that $h(\beta_1) = h(Q)\geq (D-1-O(\delta))^2 h(\tilde{R})$ by \eqref{height of tilde R is small compared to Q}, so that our bound reads $$\ind(p,\vec{\beta})\leq \frac{d_2}{d_1} + \frac{(1 + \frac{d_2}{d_1})}{(c^{-2} - 1)(D-1)^2} + O(\delta),$$ as desired.
\end{proof}

Therefore there are $a,b$ such that $(\d_x^a\d_y^b p)(\beta_1, \beta_2)\neq 0$ and $$\frac{a}{d_1} + \frac{b}{d_2}\leq \frac{d_2}{d_1} + \frac{(1 + \frac{d_2}{d_1})}{(c^{-2} - 1)(D-1)^2} + O(\delta).$$ Let now $$q(x,y) := \frac{(\d_x^a\d_y^b p)(x,y)}{a!b!}\in \Z[x,y].$$ Notice that $$H(q)\leq O(1)^{d_1+d_2}H(p)\leq O(H(\alpha))^{\frac{d_1 + d_2}{c^{-2} - 1 - O(\delta)}}$$ as well. Moreover
\begin{align*}
\ind(q,(\alpha,\alpha))&\geq \ind(p,(\alpha,\alpha)) - \frac{a}{d_1} - \frac{b}{d_2}\\&\geq t - \frac{d_2}{d_1} - \frac{(1 + \frac{d_2}{d_1})}{(c^{-2} - 1)(D-1)^2} + O(\delta).
\end{align*}

Let now $k_*,\ell*\geq 1$ be such that $$\frac{k_*-1}{d_1} + \frac{\ell_*}{d_2}, \frac{k_*}{d_1} + \frac{\ell_* - 1}{d_2}\leq \ind(q,(\alpha,\alpha))$$ but $$\frac{k_*}{d_1} + \frac{\ell_*}{d_2} > \ind(q,(\alpha,\alpha)).$$ Then observe that $$q(x,y) = \int_\alpha^x\cdots \int_\alpha^{w_{k_*-1}}\int_\alpha^y\cdots \int_\alpha^{z_{\ell_*-1}} (\d_x^{k_*}\d_y^{\ell_*} q)(w_{k_*},z_{\ell_*}) dw_1\cdots dw_{k_*} dz_1\cdots dz_{\ell_*}.$$ Therefore $$|q(x,y)|\leq |x-\alpha|^{k_*}|y-\alpha|^{\ell_*} \sup_{(w,z)\in [\alpha,x]\times [\alpha,y]} \left|\frac{(\d_x^{k_*}\d_y^{\ell_*} q)(w,z)}{k_*!\ell_*!}\right|.$$ Hence $q(\beta_1,\beta_2)\neq 0$ is bounded above in absolute value by $$|q(\beta_1, \beta_2)|\leq H(\beta_1)^{-\kappa k_*}H(\beta_2)^{-\kappa \ell_*} O(H(\alpha))^{\frac{d_1 + d_2}{c^{-2} - 1 - O(\delta)}} O(T + |x(R_{\vec{\alpha}})|)^{d_1 + d_2}.$$ But it is also a nonzero rational with denominator at most $H(\beta_1)^{d_1}H(\beta_2)^{d_2}$, so that $$|q(\beta_1, \beta_2)|\geq H(\beta_1)^{-d_1}H(\beta_2)^{-d_2}.$$
Therefore (using $d_1h(\beta_1) = d_2h(\beta_2) + O\left(\frac{h(\beta_2)}{d_2}\right)$) we have derived the inequality
\begin{align*}
-2d_1h(\beta_1)\leq -\kappa d_1h(\beta_1)\left(\frac{k_*}{d_1} + \frac{\ell_*}{d_2}\right) + \frac{(d_1 + d_2)h(\alpha)}{c^{-2} - 1 - O(\delta)} + (d_1 + d_2)\max(\log{T}, \log{|x(R_{\vec{\alpha}})|}) + O(d_1 + d_2).
\end{align*}
Using $$\frac{k_*}{d_1} + \frac{\ell_*}{d_2} > \frac{\sqrt{2}c}{3} - \frac{d_2}{d_1} - \frac{(1 + \frac{d_2}{d_1})}{(c^{-2} - 1)(D-1)^2} + O(\delta)$$ and dividing through by $d_1h(\beta_1)$, we find that
\begin{align*}
\left(\frac{\sqrt{2}c}{3} - \frac{d_2}{d_1}\right)\kappa - \left(\frac{(1 + \frac{d_2}{d_1})}{(c^{-2} - 1)(D-1)^2}\right)(\kappa + 1) - \frac{9(1 + \frac{d_2}{d_1})}{(D-1)^2} < 2 + O(\delta).
\end{align*}

Finally, recall that $\frac{d_2}{d_1} = \frac{h(\beta_1)}{h(\beta_2)} + O\left(\frac{1}{d_1^2}\right)\leq (\kappa - 1)^{-s} + O(\delta)$. Inserting this into the bound we get that
\begin{align*}
\left(\frac{\sqrt{2}c}{3} - \frac{1}{(\kappa - 1)^s}\right)\kappa - \frac{1 + \frac{1}{(\kappa - 1)^s}}{(D-1)^2}\left(9 + \frac{\kappa + 1}{(c^{-2} - 1)}\right) < 2 + O(\delta).
\end{align*}

This contradicts the hypothesis once $\delta\ll_{c,D} 1$, and so we are done.
\end{proof}

\subsection{Conclusion of proof}

Summarizing, we have proved:
\begin{prop}
Let $c < 1$, $D > 1$, $\tilde{D} := \frac{D + \sqrt{D^2+4}}{2}$, $C := 5\tilde{D}^2$, and $s\in \Z^+$ be such that $$\frac{576}{C} + \frac{72}{D^2} + \max\left(\frac{19}{C}, \frac{19}{D^2}\right) < \frac{1}{2}$$ and $$\left(\frac{\sqrt{2}c}{3} - \frac{1}{(\kappa - 1)^s}\right)\kappa - \frac{1 + \frac{1}{(\kappa - 1)^s}}{(D-1)^2}\left(9 + \frac{\kappa + 1}{(c^{-2} - 1)}\right) > 2,$$ where $$\kappa := \left(\frac{9}{2} - \max\left(\frac{171}{C}, \frac{171}{D^2}\right) - \frac{504}{C} - \frac{63}{D^2}\right)\left(1 + D^{-1}\right)^{-2}.$$

Let $\delta\ll_{c,D} 1$. Let $T\gg_{c,D,\delta} 1$. Let $1 < J < 2$. Let $E\in \mathcal{F}_*$. Then: 
\begin{enumerate}
\item If $\rank(E) = 0$ then $\#|E(\Z)| = 0$. 
\item If $\rank(E) = 1$ then $\#|E(\Z)|\leq 2$.
\item If $\rank(E) = r > 1$ then:
\begin{align*}
\#|E(\Z)|&\leq 2r \left\lceil\frac{\log{\tilde{D}}}{\log{J}}\right\rceil\cdot \max_{S\subseteq \RP^{r-1} : \forall v\neq w\in S, |\langle v, w\rangle|\leq \frac{J}{2} + O(\delta)} \#|S|\\&\quad + 9s(3^r - 1).
\end{align*}
\end{enumerate}
\end{prop}

Note that, of course, this implies that if the density of curves with ranks $0$ and $1$ are both $\frac{1}{2}$, then $\limsup_{T\to\infty}\Avg_{\FuniversalT}(\#|E(\Z)|)\leq 2$, as claimed. (To see this, the only question is the contribution from the density zero higher-rank curves. To bound this, use the proposition and the Kabatiansky-Levenshtein bound (Theorem \ref{kablev}) and then combine H\"{o}lder's inequality with Bhargava-Shankar as usual.)

In any case, the details of the optimization procedure given this bound are given in the appendix since the rest of the argument is unrelated to Diophantine geometry. This completes the argument.
\end{proof}

\section{Proof of Theorem \ref{amazing theorem} and its corollaries}\label{amazing theorem proof}

To get inexplicit bounds we may simply follow the general procedure of the proof of Theorem \ref{constant theorem}. On examination, to prove Theorem \ref{amazing theorem} for a family $\mathcal{F}$, it is clear that the only estimates required are:
\begin{enumerate}
\item An estimate on small points: $$\limsup_{T\to\infty} \Avg_{\mathcal{F}^{\leq T}}\left(\#|\{P\in E(\Z) : h(P)\leq C\log{T} + O(1)\}|\right)\ll 1,$$
\item and a repulsion estimate on larger points: if $P\neq R\in E(\Z)$ with $h(P), h(R)\geq C\log{T} + O(1)$ and $h(R)\leq h(P)\leq (1+\Omega(1))h(R)$, then $$\cos{\theta_{P,R}}\leq 0.88,$$
\end{enumerate}
where the $0.88$ has come from the Kabatiansky-Levenshtein bound (Theorem \ref{kablev}) --- specifically, the solution to $$\exp\left(\frac{1+\sin{\theta}}{2\sin{\theta}}\log{\left(\frac{1+\sin{\theta}}{2\sin{\theta}}\right)} - \frac{1-\sin{\theta}}{2\sin{\theta}}\log{\left(\frac{1-\sin{\theta}}{2\sin{\theta}}\right)}\right) = 3$$ has $\cos{\theta} = 0.898...$.

From there one bounds the small points by the first part, the medium points by projecting those in an interval of shape $[X,(1+\Omega(1))X]$ to the unit sphere and applying Kabatiansky-Levenshtein, and the large points by using Siegel's argument, exactly as we did in the proof of Theorem \ref{constant theorem}. So to prove Theorem \ref{amazing theorem} we will provide exactly these ingredients. Since the families will be getting thinner and thinner (from $\asymp T^5$ for $\FuniversalT$ to $\asymp T^3$ for $\FAT$ to $\asymp T^2$ for $\FBT$ to $\asymp T$ for $\FcongT$), our constants $C$ in the small points esimates will get worse and worse (in fact we will always have $C = \frac{\log{\left(\#|\mathcal{F}|\right)}}{\log{T}}$\footnote{However, for congruent number curves, Le Boudec \cite{leboudec} has obtained a bound with $C = 2$, which is much stronger than the $C = 1$ we get with our methods.}), which will lead us to be a bit cleverer with our repulsion estimates each time. Note that the main issue in establishing the repulsion estimate is that the discriminants of the curves in these families are nowhere near squarefree, so the methods that allowed us to treat the canonical and Weil heights as roughly the same in the proof of Theorem \ref{constant theorem} do not apply here.\footnote{As a sidenote, one could also proceed by noting that the curves in each of these families are all twists of one another, and then estimating effects of twisting on the heights precisely. This reduces to a roughly similar computation, though we proceed via local heights in order to also introduce the idea of establishing repulsion between $2P$ and $2R$ for integral points $P$ and $R$. In fact, at least for the families $y^2 = x^3 + Ax$ and $y^2 = x^3 - D^2 x$, since we have such good control on the ranks of the curves in these families one could also simply apply the theorem of Hindry-Silverman (Theorem \ref{hindrysilvermanszpirobound}) after throwing out those curves with large Szpiro ratio, but the implied constants would be tremendous.}

\subsection{$y^2 = x^3 + B$}\label{mordell section}

\begin{proof}[Proof of Theorem \ref{amazing theorem} for $\mathcal{F}_{A=0}$.]
Of course $$\#|\FAT|\asymp T^3.$$

To count points with $|x|\leq 10^{10}T$, note that $y\ll T^{\frac{3}{2}}$ and that $x$ and $y$ determine $B$. Therefore the number of solutions $(x,y,B)$ with $|B|\ll T^3$ and $|x|\leq 10^{10}T$ is at most the number of $|x|\leq 10^{10}T$ and $|y|\ll T^{\frac{3}{2}}$, which is $\ll T^{2.5}$.

Otherwise, $|y|\asymp |x|^{\frac{3}{2}}$. But now given an $x$, if $(x,y,B)$ and $(x,y',B')$ are both solutions and (without loss of generality) $y,y' > 0$, then $$y^2 - y'^2 = B - B',$$ whence $$|y - y'|\ll \frac{T^3}{|x|^{\frac{3}{2}}}.$$ Hence, given an $10^{10}T\leq |x|\ll T^3$, the number of $y$ such that $|x^3 - y^2| =: |B|\ll T^3$ is $$\ll 1 + \frac{T^3}{|x|^{\frac{3}{2}}}.$$

Therefore, taking these together, the number of solutions $(x,y,B)$ with $|x|\ll T^3$ is $$\ll T^{2.5} + \sum_{T\ll |x|\ll T^3} 1 + \frac{T^3}{|x|^{\frac{3}{2}}}\ll T^3.$$ This contributes $\ll 1$ to the average.

So we have proved the first necessary result. For the second, again restrict (by Helfgott-Venkatesh, H\"{o}lder, and now Fouvry \cite{fouvry} instead of Bhargava-Shankar) to the subfamily with the largest square divisor of $B$ at most $\ll T^\delta$ and with $|\Delta|\asymp |B|^2\gg T^{6-\delta}$. Now $j(E_{0,B}) = 0$, so that, by Lang \cite{lang} (Chapter III, Section 4), at $p > 3$ such that $v_p(B) = 1$, $$\lambda_p(Q) - \hat{\lambda}_p(Q) = \log^+{|x(Q)|_p} - \log^+{|B^{-\frac{1}{3}}x(Q)|_p},$$ where $\lambda_p$ and $\hat{\lambda}_p$ are the local heights for $h$ and $\hat{h}$, respectively, and we have written $Q$ for a rational point on $E$. (Note that Lang's normalizations are different from ours by a factor of $2$.)

Now this expression for $\lambda_p - \hat{\lambda}_p$ is $\frac{1}{3}\log{|B|_p}$ unless $v_p(x(Q))\geq \frac{1}{3}v_p(B) = \frac{1}{3}$ --- i.e., unless $v_p(x(Q))\geq 1$. But $$y(Q)^2 = x(Q)^3 + B$$ and $v_p(x(Q))\geq 1, v_p(B) = 1$ implies $v_p(y(Q))\geq 1$, so $v_p(B)\geq 2$, a contradiction. Thus this expression is always equal to $$\frac{1}{3}\log{|B|_p}$$ when $v_p(B) = 1$.

For primes such that $p^2\vert B$ (or $p = 2,3$), Lang also proves that $$|\lambda_p(x(Q)) - \hat{\lambda}_p(x(Q))|\ll -\log{|\Delta|_p}.$$ Thus the sum over these primes is $O(\delta\log{T})$.

Finally, for the infinite prime, Lang proves that $$\lambda_\infty(Q) - \hat{\lambda}_\infty(Q) = \log^+{|x(Q)|} - \log^+{\left(|\Delta|^{-\frac{1}{6}}|x(Q)|\right)}.$$

Therefore, exactly as before, $$\hat{h}(Q) - h(Q) = \log^+{\left(|\Delta|^{-\frac{1}{6}}|x(Q)|\right)} + \frac{1}{6}\log^+{|\Delta|} - \log^+{|x(Q)|} + O(\delta\log{T})$$ by the product formula. Therefore we may simply repeat the proof of Lemma \ref{helfgott-mumford gap principle} verbatim. This completes the ingredients necessary for this family.
\end{proof}

\subsection{$y^2 = x^3 + Ax$}\label{kane thorne curves section}

Now let us move to the family $y^2 = x^3 + Ax$.

\begin{proof}[Proof of Theorem \ref{amazing theorem} for $\mathcal{F}_{B=0}$.]
The family is of size $$\#|\FBT|\asymp T^2.$$

Fixing $y$, since $x, x^2 + A$ are both divisors of $y^2$, the number of $(x,A)$ pairs such that $(x,y,A)$ is a solution is at most the number of pairs of divisors $(d_1, d_2)$ of $y^2$, which is $\tau(y^2)^2\ll_\eps y^\eps$. Therefore the number of $(x,y,A)$ such that $|x|\leq T^{\frac{4}{3} - \eps}$ and $|A|\ll T^2$ is at most (since $|y|\ll T^{2 - \frac{3}{2}\eps}$ in this case) $$\ll \sum_{|y|\ll T^{2 - \frac{3}{2}\eps}} y^{\frac{\eps}{2}}\ll T^{2 - \frac{\eps}{2}}.$$

For $|x|\geq T^{\frac{4}{3}-\eps}$ (so that $|y|\asymp |x|^{\frac{3}{2}}$), fix $x$ and note that if $(x,y,A)$ and $(x,y',A')$ are both solutions and $y, y' > 0$ without loss of generality, then $$y^2 - y'^2 = x(A-A'),$$ so that $$|y-y'|\ll \frac{T^2}{|x|^{\frac{1}{2}}}.$$ Thus all the $|y|$ live in an interval of length $$\ll \frac{T^2}{|x|^{\frac{1}{2}}}.$$ Note also that $y^2\equiv 0\pmod{x}$, which has $\prod_{p\vert x} p^{\floor{\frac{v_p(x)}{2}}}$ solutions modulo $x$. Therefore the number of $y$ given $x$ is at most $$\ll 1 + \frac{T^2\cdot \prod_{p\vert x} p^{\floor{\frac{v_p(x)}{2}}}}{|x|^{\frac{3}{2}}}.$$

Thus, taking these together, the number of solutions $(x,y,A)$ with $|x|\ll T^2$ and $|A|\ll T^2$ is at most
\begin{align*}
&\ll T^{2 - \frac{\eps}{2}} + \sum_{T^{\frac{4}{3} - \eps}\ll |x|\ll T^2} \left(1 + \frac{T^2\cdot \prod_{p\vert x} p^{\floor{\frac{v_p(x)}{2}}}}{|x|^{\frac{3}{2}}}\right)
\\&\ll T^2.
\end{align*}

Thus we have counted small points. For the second ingredient, we would again (with more difficulty) be able to prove a repulsion bound in terms of $h(P)$ and $h(R)$, but estimating the error in this bound for points of small height would give us serious difficulty. Moreover, these methods would not work for the next case where we restrict to square $A$. So we introduce another idea.

First restrict to $A$ that have largest square divisor at most $T^\delta$ and such that $|A|\geq T^{2-\delta}$.\footnote{To do this, see Lemma \ref{better pointwise bound for B=0}.} Note that $j(E_{A,0}) = 1728\in \Z$, so that again Lang applies, whence once $p > 3$ and $v_p(A) = 1$ the difference of local Weil and canonical heights is $$\lambda_p(Q) - \hat{\lambda}_p(Q)  = \log^+{|x(Q)|_p} - \log^+{|A^{-\frac{1}{2}}x(Q)|_p},$$ which is $$\frac{1}{2}\log{|A|_p}$$ unless $$v_p(x(Q))\geq \frac{1}{2}v_p(A) = \frac{1}{2}$$ --- i.e., unless $v_p(x(Q))\geq 1$. In this case the expression is $0$, which we will write as $$\frac{1}{2}\log{|A|_p} - \frac{1}{2}\log{|A|_p}.$$ Again, at $p = 2,3$ or $p$ such that $p^2\vert A$, the difference of local heights is $\ll -\log{|\Delta|_p}$. At the infinite place, as before the contribution to the difference is $$\log^+{|x(Q)|} - \log^+{\left(|A^{-\frac{1}{2}}|\cdot |x(Q)|\right)}.$$

Therefore we have found that (applying the product formula as before)
\begin{align}\label{canonical in terms of weil}
\hat{h}(Q) - h(Q) = \log^+{\left(|\Delta|^{-\frac{1}{6}}\cdot |x(Q)|\right)} + \frac{1}{6}\log{|\Delta|} - \log^+{|x(Q)|} + \frac{1}{2}\sum_{p\vert\vert A, v_p(x(Q))\geq 1} \log{|A|_p} + O(\delta\log{T}).
\end{align}\noindent
Since the $\log{|A|_p}$ terms are simply $-\log{p}$, this gives us a way of getting an upper bound on $\hat{h}$: $$\hat{h}(Q) - h(Q)\leq \log^+{\left(|\Delta|^{-\frac{1}{6}}\cdot |x(Q)|\right)} + \frac{1}{6}\log{|\Delta|} - \log^+{|x(Q)|} + O(\delta\log{T}).$$

Now for the new idea. Let $P\neq \pm R\in E(\Z)$ with $h(P)\geq h(R)\geq 2\log{T}$. Write instead $$\cos{\theta_{P,R}} = \frac{\hat{h}(2P+2R) - \hat{h}(2P) - \hat{h}(2R)}{2\sqrt{\hat{h}(2P)\hat{h}(2R)}}.$$ From the above we have the upper bound $\hat{h}(2P+2R)\leq h(2P+2R) + O(\delta\log{T})$.

Moreover, writing $P =: (x,y)$, if $p\vert\vert A$ and $v_p(x(2P))\geq 1$, then since $$x(2P) = \frac{(3x^2 + A)^2}{4y^2} - 2x = \frac{x^4 - 2Ax^2 + A^2}{4y^2},$$ we see that $p\vert x$. But then $p\vert y$ since $y^2 = x^3 + Ax$. Hence $p^2\vert 4y^2$. Since $v_p(x(2P))\geq 1$, we see that $p^3\vert x^4 - 2Ax^2 + A^2$, whence $p^3\vert A^2$, which is to say $p^2\vert A$, a contradiction. The same holds for $R$, so that we have found (by \eqref{canonical in terms of weil}) that $$\hat{h}(2P) = h(2P) + O(\delta\log{T})$$ and $$\hat{h}(2R) = h(2R) + O(\delta\log{T}).$$ Also note that $h(2P)\leq 4h(P) + O(1)$ since the expression for $x(2P)$ has numerator at most $O(x^4)$ and denominator at most $O(x^3)$, and upon cancelling common terms these estimates still hold.

Finally, let us write out $x(2P + 2R)$ in terms of $x(2P)$ and $x(2R)$. Write $2P =: \left(\frac{\alpha}{\beta^2}, \frac{\tilde{\alpha}}{\beta^3}\right)$ and $2R =: \left(\frac{\alpha'}{\beta'^2}, \frac{\tilde{\alpha}'}{\beta'^3}\right)$. Recall that $|x(2P)|\asymp |x(P)|$ and similarly for $R$ since $|x(P)|, |x(R)|\geq 10^{10}T$. Thus certainly $|\alpha|\geq |\beta|^2$ and similarly for $R$, so that $H(P) = |\alpha|$ and $H(R) = |\alpha'|$. Moreover for the same reason $|y(2P)|\asymp |x(2P)|^{\frac{3}{2}}$, so that $|\tilde{\alpha}|\asymp |\alpha|^{\frac{3}{2}} = H(\alpha)^{\frac{3}{2}}$ and similarly for $R$.

Now
\begin{align*}
x(2P + 2R) &= \frac{x(2P)^2 x(2R) + x(2P) x(2R)^2 + 2 y(2P) y(2R) + A x(2P) + A x(2R)}{(x(2P) - x(2R))^2}
\\&= \frac{\alpha^2 \alpha' \beta'^2 + \alpha \alpha'^2 \beta^2 + 2\tilde{\alpha}\tilde{\alpha}'\beta\beta' + A\alpha\beta^2\beta'^4 + A\alpha'\beta^4\beta'^2}{(\alpha - \alpha')^2}.
\end{align*}
By using the first expression and the fact that $|x(2P)|\asymp |x(P)|$ (and similarly for $R$) it follows that the first term in the numerator is the largest (up to $O(1)$) among those in the numerator or denominator since $h(P)\geq h(R)$. Therefore
\begin{align*}
H(2P+2R)&\ll |\alpha|^2 |\alpha'| |\beta'|^2
\\&= \frac{H(2P)^2 H(2R)^2}{|x(2R)|}
\\&\asymp \frac{H(2P)^2 H(2R)^2}{|x(R)|}
\\&= \frac{H(2P)^2 H(2R)^2}{H(R)},
\end{align*}
which is to say $h(2P + 2R)\leq 2h(2P) + 2h(2R) - h(R) + O(1)$. Since $4h(R)\geq h(2R) - O(1)$, this reduces to $$h(2P + 2R)\leq 2h(2P) + \frac{7}{4}h(2R) + O(1).$$

Therefore, putting these together and arguing as in Lemma \ref{helfgott-mumford gap principle}, we find that
\begin{align*}
\cos{\theta_{P,R}} &= \frac{\hat{h}(2P+2R) - \hat{h}(2P) - \hat{h}(2R)}{2\sqrt{\hat{h}(2P)\hat{h}(2R)}}
\\&\leq \frac{\hat{h}(2P+2R) - h(2R) - h(2R) + O(\delta\log{T})}{2\sqrt{h(2P)h(2R)}}
\\&\leq \frac{1}{2}\sqrt{\frac{h(2P)}{h(2R)}} + \frac{3}{8}\sqrt{\frac{h(2R)}{h(2P)}} + \frac{O(\delta\log{T})}{2\sqrt{h(2P)h(2R)}}.
\end{align*}

Now suppose we could show any nontrivial (i.e., not $x = 0$) rational point on $y^2 = x^3 + Ax$ must have height at least $c\log{T}$ for some $c\gg 1$ a (very small) positive constant. Then this upper bound would read:
\begin{align*}
\cos{\theta_{P,R}}\leq \frac{1}{2}\sqrt{\frac{h(2P)}{h(2R)}} + \frac{3}{8}\sqrt{\frac{h(2R)}{h(2P)}} + O(\delta).
\end{align*}

Hence this would complete the proof of the second necessary ingredient, since $\frac{1}{2} + \frac{3}{8} = \frac{7}{8} = 0.875 < 0.88$. This is because the number of points $P$ with $h(P)\geq 2\log{T} + O(1)$ and $h(2P)\in [X,(1+\gamma)X]$ is then $$\ll \gamma^{-1}\cdot 3^{\rank(E)}$$ once $\gamma\ll 1$. Hence since $$h(2P)\geq c\log{T}\gg \log{T},$$ the number of points $P$ with $h(P)\geq 2\log{T} + O(1)$ and $h(2P)\leq M\log{T}$ is $$\ll \log{(M)}\cdot 3^{\rank(E)}.$$ But $h(P)\geq \frac{1}{4} h(2P) - O(1)$, so the number of points with $h(P)\leq M\log{T}$ is in fact also $\ll \log{(M)}\cdot 3^{\rank(E)}$, which is all we need to conclude the argument.

Thus it suffices to show that the smallest nontrivial rational point has height at least $c\log{T}$ for some positive $c\gg 1$. Actually it suffices to do this for a large enough subfamily of curves, by the usual H\"{o}lder, Helfgott-Venkatesh, and then Bhargava-Shankar-type procedure.\footnote{Again, see Lemma \ref{better pointwise bound for B=0} for details.} We will show that the density of curves with a nontrivial rational point of multiplicative height smaller than $T^{\frac{1}{100}} =: T^c$ is $T^{-\Omega(1)}$.

Now if $\left(\frac{m}{n^2}, \frac{m'}{n^3}\right)$ is a point on $y^2 = x^3 + Ax$ with $|m|\leq T^c, |n|\leq T^{\frac{c}{2}}$, then $(m,m')$ is an integral point on $y^2 = x^3 + An^4 x$ with $|m|\leq T^c$. Note that $|m'|\ll T^{1+\frac{3}{2}c}$. The number of such pairs $(m,m')$ is at most $T^{1+3c}$. Moreover since $(m,m')$ determine $A$ and $n$ (up to sign) since $A$ is fourth-power free by minimality, we see that the number of $A$ with $E_{A,0}$ with a nontrivial rational point of height at most $T^c$ is at most the number of such rational points on an $E_{A,0}$ for some $A$, which is at most the number of $(m,m')$ pairs, which is at most $T^{1+3c}$. Thus the density is $T^{-1+3c}$, which is of the desired shape.

This completes the argument.
\end{proof}

Having proven this, let us now explain how to use the methods of Kane \cite{kane} and Kane-Thorne \cite{kanethorne} to deduce Corollary \ref{B=0 theorem}. We will freely use their notation throughout, and for ease of reading one should at least go through their arguments to understand the effects of our modifications.

\begin{proof}[Proof of Corollary \ref{B=0 theorem}]
Let us quickly show that to control an average of e.g.\ $2^{k\cdot \rank(E)}$ it suffices to control moments of Selmer groups on the curves. Let $\phi_A: E_{A,0}\to E_{-4A,0}$ be the $2$-isogenies on the curves. Let $\Sel_{\phi_A}(E_{A,0})$ be the associated Selmer groups. Note that the isogeny dual to $\phi_A$ is simply $\phi_{-4A}$. Hence $\phi_{-4A}\circ \phi_A = 2\cdot$, multiplication by $2$ on $E_{A,0}$. The following Lemma (combined with Cauchy-Schwarz) shows that to control the average of $2^{k\cdot \rank(E)}\leq \#|\Sel_2(E)|^k$ it is enough to control the moments of $\#|\Sel_{\phi_A}(E_{A,0})|$.

\begin{lem}\label{selmer groups lemma}
Let $E\xrightarrow{\alpha} E'\xrightarrow{\beta} E''$ be a sequence of isogenies between elliptic curves over $\Q$. Then $$\#|\Sel_{\beta\circ \alpha}(E)|\leq \#|\Sel_\alpha(E)|\cdot \#|\Sel_\beta(E')|.$$
\end{lem}
\begin{proof}
Consider the long exact sequence in Galois cohomology associated to $0\to \ker{\alpha}\to \ker{(\beta\circ \alpha)}\to \ker{\beta}\to 0$. It induces a sequence $\Sel_\alpha(E)\to \Sel_{\beta\circ\alpha}(E)\to \Sel_\beta(E')$ which is exact at the middle term. (Surjection onto the kernel follows from exactness on $H^1$ and the fact that the left-hand map is induced by the identity map $E\to E$ so only locally trivial classes map to one another.) The result follows.
\end{proof}\noindent
Hence we will concentrate on bounding moments of $\#|\Sel_{\phi_A}(E_{A,0})|$, as Kane-Thorne do.

The next claim is that for this family we may improve Lemma \ref{using the pointwise bound} to:
\begin{lem}\label{better pointwise bound for B=0}
Let $\mathcal{G}\subseteq \mathcal{F}_{B=0}^{\leq T}$. Then, for all $\eps > 0$, $$\sum_{E\in \mathcal{G}} \#|E(\Z)|\ll \#|\mathcal{F}_{\mathrm{B=0}}^{\leq T}|\cdot \left(\frac{\#|\mathcal{G}|}{\#|\mathcal{F}_{\mathrm{B=0}}^{\leq T}|}\right)^{\Omega(1)}\cdot (\log{T})^{O(1)}.$$
\end{lem}
\begin{proof}
The only change in the proof of Lemma \ref{using the pointwise bound} is that $\omega(\Delta)$ is replaced by $\omega(A)$ and now we may use the bound $\rank(E_{A,0})\ll \omega(A)$ as well (this comes from a descent by $2$-isogeny: see Proposition 4.9 in Chapter X, Section 4 of \cite{silvermanarithmeticofellipticcurves}). Instead of using the bound $\omega(\Delta)\ll \frac{\log{T}}{\log\log{T}}$, we instead use $\sum_{n\leq X} O(1)^{\omega(n)}\ll X(\log{X})^{O(1)}$.
\end{proof}

Hence we may restrict to a subfamily of density $1 - O\left((\log{T})^{-M}\right)$ once $M\gg 1$. Hence we may further impose the restriction that $\omega(A)\leq M\log\log{A}$ for some sufficiently large constant $M$ on our curves (on top of the usual restriction that $A$ have non-squarefree part at most $T^\delta$), since the number of $n\leq X$ with $m$ prime factors is at most $$\ll \frac{X}{\log{X}}\cdot \frac{(\log\log{X}+O(1))^m}{m!}.$$ Moreover, suppose there is a real character $\chi$ of modulus $D\ll T$ with $L(s,\chi)$ having a real zero $\beta_\chi$ with $1 - \beta_\chi\leq (\log{T})^\delta$. Then since (by Siegel's theorem on Siegel zeroes) $1 - \beta_\chi\gg_\eps D^{-\eps}$ for all $\eps > 0$, we find that $D\gg_\delta (\log{T})^{M+1}$, for instance. Hence once $T\gg_\delta 1$ (with ineffective implied constant) we may remove all $A$ divisible by $D$ as well. As Kane notes on page 17 of \cite{kane}, this implies $1 - \beta_\chi\gg (\log{T})^{-1}$ for any real zeroes $\beta_\chi$ of $L(s,\chi)$ with $\chi$ of modulus not divisible by $D$ and at most $T$.

Call the resulting subfamily $\widetilde{\mathcal{F}}_{B=0}\subseteq \mathcal{F}_{B=0}$. Let us now indicate the necessary changes to Kane's argument in \cite{kane} in order to get a bound of shape $$\limsup_{T\to\infty}\Avg_{E\in \widetilde{\mathcal{F}}_{B=0}^{\leq T}}(k^{\rank(E)})\ll O(1)^{(\log{k})^2}.$$ We first fix a positive integer $F\leq T^\delta$ such that $p\vert F\implies p^2\vert F$ for all primes $p > 2$ and restrict our attention to the subfamily of $D$ with $F = 2^{v_2(D)}\mathrm{sq}(D) := 2^{v_2(D)}\prod_{p^2\vert D : p > 2} p^{v_p(D)}$. The claim is that the restrictions $\frac{\log\log{N}}{2} < n < 2\log\log{N}$ may be replaced by $n < M\log\log{N}$, where $M$ is the sufficiently large constant arising in the definition of $\widetilde{\mathcal{F}}_{B=0}$. To prove this, we change the following in Kane's argument. In Proposition 11 we replace $O\left(\frac{N}{\sqrt{\log\log{N}}}\right)$ by $\max_{\tilde{n}\leq n} \pi_{\tilde{n}}(N)$, where $\pi_{\tilde{n}}(N)$ is the number of integers in $[1,N]$ with exactly $\tilde{n}$ prime factors. This improves Lemma 17 to a bound of shape $$\ll \max_{\tilde{n}\leq n} \pi_{\tilde{n}}(N)\cdot \left(\left(\frac{O(\log\log{B})}{L}\right)^k + \cdots\right).$$ In the proof of Proposition 9 we instead obtain a bound of shape $$\ll O(1)^k\cdot \left(\max_{\tilde{n}\leq n} \pi_{\tilde{n}}(N)\right)\cdot \left(\left(\frac{\eps\log\log{N}}{n}\right)^k + (\log{N})^{-C}\right).$$

If $n\gg \log\log\log{N}$ and $N\gg_{c,k} 1$, then this is $\ll N\cdot c^m$, as in Kane. If $n\ll \log\log\log{N}$, then
\begin{align*}
\max_{\tilde{n}\leq n} \pi_{\tilde{n}}(N) = \pi_n(N)&\asymp \frac{(\log\log{N})^n}{n!}\cdot \frac{N}{\log{N}} \\&\ll \frac{N}{\log{N}}\cdot O(1)^{(\log\log\log{N})^2}.
\end{align*}
Hence the resulting bound in this case is
\begin{align}
&\ll \frac{N}{\log{N}}\cdot O(1)^{(\log\log\log{N})^2}\left(\left(\frac{\log\log{N}}{n}\right)^k + 1\right)\nonumber
\\&\ll \frac{N}{\log{N}}\cdot O(1)^{(\log\log\log{N})^2},\label{very few prime factors bound}
\end{align}\noindent
since $k\leq n\ll \log\log\log{N}$. This is again $\ll N\cdot c^n\ll N\cdot c^m$ once $N\gg_c 1$ since $N\cdot c^n\gg N(\log\log{N})^{-O(\log{c})}$.

Thus we have the necessary improvement to Kane's Proposition 9 to feed into the analysis in Kane-Thorne. As they note, the contribution of terms with $m > 0$ is (once $N\gg_k 1$ and e.g.\ $c = 2^{-2k-1}$)
\begin{align*}
&\ll_k N 2^{-kn} \sum_{m=1}^n {n\choose m} (2^k-1)^{n-m} 4^{km} c^m \omega(F)
\\&\leq N\cdot (1-2^{-k-1})^n \omega(F)
\\&\leq N\cdot (\log{N})^{-\Omega(2^{-k})}\cdot \omega(F)
\end{align*}\noindent
if $n\gg \log\log\log{N}$. When $n\ll \log\log\log{N}$ we use the stronger bound in \eqref{very few prime factors bound} to obtain
\begin{align*}
&\ll_k \frac{N}{\log{N}} O(1)^{(\log\log\log{N})^2} 2^{-kn} \sum_{m=1}^n {n\choose m} (2^k-1)^{n-m} 4^{km} \omega(F)
\\&\leq \frac{N}{\log{N}} O(1)^{(\log\log\log{N})^2} O(1)^{kn} \omega(F)
\\&\ll N\cdot (\log{N})^{2^{-1}}\cdot \omega(F)
\end{align*}\noindent
once $N\gg_k 1$. So we may ignore the terms with $m > 0$.

Also, as in Kane-Thorne, the sum over terms with $m=0$ is $$\ll O(1)^{k^2}\cdot O(1)^{\omega(F)}\cdot \#|\{|x|\leq N : F\vert x, \omega(x) = n, 2^{v_p(x)}\mathrm{sq}(x) = F\}|,$$ where $\mathrm{sq}(x)$ is the ``odd squarefull'' part of $x$: $$\mathrm{sq}(x) = \prod_{p^2\vert x : p>2} p^{v_p(x)}.$$

Summing over all $n\ll \log\log{N}$, we find that the sum of $2^{k\cdot \rank(E)}$ over those $E$ with $2^{v_2(D)}\prod_{p^2\vert D} p^{v_p(D)} = F$ is
\begin{align*}
&\ll O(1)^{k^2}\cdot O(1)^{\omega(F)}\cdot \#|\{|x|\leq N : F\vert x, 2^{v_2(x)}\mathrm{sq}(x) = F\}|
\\&\ll O(1)^{k^2}\cdot \frac{O(1)^{\omega(F)}}{F}\cdot \#|\{|x|\leq N : 2^{v_2(x)}\mathrm{sq}(x)\leq T^\delta\}|,
\end{align*}
whence the contribution to the average of those $D$ with ``even/squarefull part'' $F$ is $\ll O(1)^{k^2}\cdot \frac{O(1)^{\omega(F)}}{F}$.

Summing over $F\leq T^\delta$ such that $p\vert F\implies p^2\vert F$ for all $p > 2$ gives the result. Indeed, $$\sum_{F\leq T^\delta: F^{\mathrm{odd}}\textrm{ squarefull}} \frac{O(1)^{\omega(F)}}{F}\ll 1.$$
\end{proof}

\subsection{$y^2 = x^3 - D^2 x$}\label{congruent number curves subsection}\ 

Finally, we handle the congruent number curves.

\begin{proof}[Proof of Theorem \ref{amazing theorem} for $\mathcal{F}_{\mathrm{congruent}}$.]
The family is of size $$\#|\FcongT|\asymp T.$$

First, the small points. We will in fact drop the restriction that $D$ be squarefree when counting the small points since it will not be necessary, but we may, and will, assume $|D|\geq T^{1-\delta}$. Fix $x\neq 0$. Break up the set of solutions $(x,y,D)$ with $y,D > 0$ and $D\neq \pm x$ (without loss of generality) into two classes: those with $|D - |x||\leq T^{\frac{1}{3}}|x|^{\frac{1}{3}}$ and those with $|D - |x|| > T^{\frac{2}{3}}$.

Let $(y,D),(y',D')$ be two solutions. As usual, by taking differences, $$|y-y'|\ll \frac{|D - D'||x|}{|y|}.$$ Now $$|x(x-D)(x+D)|\gg |x||D||D - |x||,$$ so $$|y|\gg |x|^{\frac{1}{2}}|D|^{\frac{1}{2}}|D - |x||^{\frac{1}{2}}.$$ Thus $$|y-y'|\ll \frac{|D - D'||x|^{\frac{1}{2}}}{|D - |x||^{\frac{1}{2}}|D|^{\frac{1}{2}}}.$$

Hence if $(x,y,D)$ and $(x,y',D')$ are solutions of the first class, then $D$ and $D'$ are close, so that $$|y-y'|\ll T^{\frac{1}{3}}|x|^{\frac{1}{3}}.$$ If $(x,y,D)$ and $(x,y',D')$ are solutions of the second class and $D$ is maximal among all such solutions, then $$|y-y'|\ll D^{\frac{1}{2}}|x|^{\frac{1}{2}}T^{-\frac{1}{6}}|x|^{-\frac{1}{6}}\ll T^{\frac{1}{3}}|x|^{\frac{1}{3}}.$$ Thus in general $$|y-y'|\ll T^{\frac{1}{3}}|x|^{\frac{1}{3}}.$$

Now also $y^2\equiv 0\pmod{x}$, which has $\prod_{p\vert x} p^{\floor{\frac{v_p(x)}{2}}}$ solutions modulo $x$. Therefore since the $y$ for which there exists a $D$ making $(x,y,D)$ a solution all lie in at most four intervals (depending on sign and class) of length at most $\ll T^{\frac{1}{3}}|x|^{\frac{1}{3}}$ and since $(x,y)$ determine $\pm D$, we find that there are at most $$\ll 1 + \frac{T^{\frac{1}{3}}\prod_{p\vert x} p^{\floor{\frac{v_p(x)}{2}}}}{|x|^{\frac{2}{3}}}$$ solutions with fixed $x$.

Therefore we find that the number of $(x,y,D)$ with $|x|\leq 10^{10}T$ and $|D|\ll T$ is at most $$\ll \sum_{|x|\leq 10^{10}T} 1 + \frac{T^{\frac{1}{3}}\prod_{p\vert x} p^{\floor{\frac{v_p(x)}{2}}}}{|x|^{\frac{2}{3}}}.$$

But the Dirichlet series
\begin{align*}
\sum_{n\geq 1} \frac{\prod_{p\vert n} p^{\floor{\frac{v_p(n)}{2}}}}{n^{s+\frac{2}{3}}} &= \prod_p (1 + p^{-s-\frac{2}{3}} + p^{-2s-\frac{1}{3}} + p^{-3s-1} + \cdots)
\\&= \prod_p \frac{1 + p^{-s-\frac{2}{3}}}{1 - p^{-2s-\frac{1}{3}}}
\\&= \frac{\zeta\left(2s + \frac{1}{3}\right)\zeta\left(s + \frac{2}{3}\right)}{\zeta\left(2s + \frac{4}{3}\right)^2}
\end{align*}
has its rightmost pole at $s = \frac{1}{3}$, of order two. Thus $$\sum_{n\ll T} \frac{\prod_{p\vert n} p^{\floor{\frac{v_p(n)}{2}}}}{n^{\frac{2}{3}}}\ll T^{\frac{1}{3}}\log{T},$$ whence $$\sum_{|x|\leq 10^{10}T} 1 + \frac{T^{\frac{1}{3}}\prod_{p\vert x} p^{\floor{\frac{v_p(x)}{2}}}}{|x|^{\frac{2}{3}}}\ll T + T^{\frac{2}{3}}\log{T},$$ which finishes the small point counting.\footnote{In fact, by using Proposition 1 in \cite{leboudec}, we may count small points of height $\ll T^2 (\log{T})^{-O(1)}$ instead of $\ll T$!}

Now for the repulsion estimate. The argument is exactly the same as in the case $y^2 = x^3 + Ax$ --- the only difference is that in the beginning of the argument we derive $v_p(x(Q))\geq v_p(D)$ rather than $\frac{1}{2}v_p(A)$, but we only use the consequence that this implies $v_p(x(Q))\geq 1$. The rest goes through completely, so that it suffices to show that on a density $1 - T^{-\Omega(1)}$ subfamily there are no nontrivial rational points of height smaller than $c\log{T}$ for some (small) positive constant $c$, by the same argument as in the case $y^2 = x^3 + Ax$. We will again take $c := \frac{1}{100}$.

But, as before, a rational point $\left(\frac{m}{n^2}, \frac{m'}{n^3}\right)$ with $|m|\leq T^c, |n|\leq T^{\frac{c}{2}}$ on $y^2 = x^3 - D^2 x$ corresponds to the integral point $(m,m')$ on $y^2 = x^3 - (Dn^2)^2 x$. Write $\tilde{D} := Dn^2$ --- note that the information of $\tilde{D}$ is equivalent to that of $(D,n^2)$ since $D$ is taken to be squarefree. Note also that in this case $|m'|\asymp |\tilde{D}| |x|^{\frac{1}{2}}$.

Now fix $m$. From the same argument as in the small point counting above (as we noted, we didn't need $D$ squarefree), we find that the number of $(m',\tilde{D})$ is at most $$\ll 1 + \frac{T^{\frac{1+c-\eps}{2}}\prod_{p\vert m} p^{\floor{\frac{v_p(m)}{2}}}}{|m|^{\frac{1}{2}}}.$$ Summing this up to $|m|\leq T^c$ gives a bound on the number of very small rational points on these curves of $$\ll T^c + T^{\frac{1}{2} + c - \frac{\eps}{4}},$$ which completes the argument.
\end{proof}

To deduce Corollary \ref{congruent number theorem} we will have a slightly easier time than we did for Corollary \ref{B=0 theorem}, since Heath-Brown's methods in \cite{heathbrowntwiststwo} control the moments of $2^{\rank(E)}$ over the family quite well. Again, we use his notation freely throughout, and urge the reader to go through the original argument to understand our modifications.

\begin{proof}[Proof of Corollary \ref{congruent number theorem}.]
Theorem 1 of Heath-Brown \cite{heathbrowntwiststwo} gives us the claimed bound $$\limsup_{T\to\infty}\Avg_{E\in \mathcal{F}_{\mathrm{congruent, odd}}^{\leq T}}(k^{\rank(E)})\ll O(1)^{(\log{k})^2}$$ over the subfamily $\mathcal{F}_{\mathrm{congruent, odd}}\subseteq \mathcal{F}_{\mathrm{congruent}}$ of curves $y^2 = x^3 - D^2 x$ with $D$ \emph{odd}. But extending this to $D\equiv 2\pmod{4}$ (recall $D$ is restricted to be squarefree) is no problem, since we only need an \emph{upper bound} on the average of (in Heath-Brown's notation) $2^{k\cdot s(D)}$ of shape $O(1)^{k^2}$, where $s(D)$ is the $2$-Selmer rank of $y^2 = x^3 - D^2 x$. Specifically, for these $D$ Heath-Brown's quadratic form $P$ controlling the appearance of a Legendre symbol does not change --- in fact we need only change $R$, which does not affect the shape of the upper bound.

Let us indicate the necessary changes in the argument. Lemma 1 of \cite{heathbrowntwiststwo} changes into an upper bound of shape (here $D = 2^{v_2(D)}\cdot D^{\textrm{odd}}$):
\begin{align*}
2^{s(D)}\leq \sum_{D^{\textrm{odd}} = \prod_{1\leq i\leq 4, 0\leq j\leq 4, i\neq j} D_{ij}}& \left(\frac{-1}{\alpha}\right)\left(\frac{2}{\beta}\right)\prod_{i=1}^4 4^{-\omega(D_{i0})} \prod_{0\leq j\leq 4, j\neq i} 4^{-\omega(D_{ij})} \prod_{k\neq i,j}\prod_\ell \left(\frac{D_{k\ell}}{D_{ij}}\right)\\&\cdot \left[1 + \left(\frac{2}{D_{21}D_{23}D_{31}D_{32}D_{41}D_{42}}\right) + \left(\frac{2}{D_{12}D_{14}D_{31}D_{34}D_{41}D_{43}}\right) \right.\\&\quad\left.+ \left(\frac{2}{D_{13}D_{14}D_{23}D_{24}D_{42}D_{43}}\right) + \left(\frac{2}{D_{12}D_{13}D_{21}D_{24}D_{32}D_{34}}\right)\right].
\end{align*}\noindent
The only changes required to obtain this bound are that in \cite{heathbrowntwistsone} Heath-Brown chooses a (unique) representative of a point $P\in E(\Q)/\mathrm{tors}$ with $|x|_2 = 1$ and $x>0$ --- instead one has to change the $2$-adic condition to $|x|_2 = |D|_2$. Also, instead of worrying about the condition for local solubility of the equations resulting from the $2$-descent at $p=2$ (which Heath-Brown handles by a trick reducing to Hilbert's reciprocity law), we may simply drop the condition since we are only concerned with an upper bound on $2^{s(D)}$. The rest of the argument proceeds in exactly the same way, except we trivially bound the sum remaining in Section 5 (``the leading terms'') of \cite{heathbrowntwiststwo}. This completes the proof.
\end{proof}

\appendix
\section{Optimizing the bound for Theorem \ref{constant theorem}}\label{appendix}
Let us now describe the optimization procedure for Theorem \ref{constant theorem}. Recall the explicit bound that we had proved (we have shifted $J$ by $O(\delta)$ for computational purposes below):
\begin{prop}
Let $c < 1$, $D > 1$, $\tilde{D} := \frac{D + \sqrt{D^2 + 4}}{2}$, $C:=5\tilde{D}$, and $s\in \Z^+$ be such that $$\frac{576}{C} + \frac{72}{D^2} + \max\left(\frac{19}{C}, \frac{19}{D^2}\right) < \frac{1}{2}$$ and $$\left(\frac{\sqrt{2}c}{3} - \frac{1}{(\kappa - 1)^s}\right)\kappa - \frac{1 + \frac{1}{(\kappa - 1)^s}}{(D-1)^2}\left(9 + \frac{\kappa + 1}{(c^{-2} - 1)}\right) > 2,$$ where $$\kappa := \left(\frac{9}{2} - \max\left(\frac{171}{C}, \frac{171}{D^2}\right) - \frac{504}{C} - \frac{63}{D^2}\right)\left(1 + D^{-1}\right)^{-2}.$$

Let $\delta\ll_{c,D} 1$. Let $T\gg_{c,D,\delta} 1$. Let $1 < J < 2$. Let $E\in \mathcal{F}_*$. Then: 
\begin{enumerate}
\item If $\rank(E) = 0$ then $\#|E(\Z)| = 0$. 
\item If $\rank(E) = 1$ then $\#|E(\Z)|\leq 2$.
\item If $\rank(E) = r > 1$, then:
\begin{align*}
\#|E(\Z)|&\leq 2r \left\lceil\frac{\log{\tilde{D}}}{\log{J}} + O(\delta)\right\rceil\cdot \max_{S\subseteq \RP^{r-1} : \forall v\neq w\in S, |\langle v, w\rangle|\leq \frac{J}{2}} \#|S|\\&\quad + 9s(3^r - 1).
\end{align*}
\end{enumerate}
\end{prop}

The first question is how to get an explicit bound on $\max_{S\subseteq \RP^{r-1} : \forall v\neq w\in S, |\langle v, w\rangle|\leq \frac{J}{2}} \#|S|$ for $r$ very large. (Kabatiansky-Levenshtein gives an asymptotic, but this is not enough.) Since we can take $r$ extremely large (e.g., $r\geq 13$) and Bhargava-Shankar guarantee that the proportion of curves with rank at least $r$ is $\ll 5^{-r}$, the following simpleminded estimate will suffice.
\begin{lem}
Let $\theta_0 > 0$, let $r\geq 3$, and let $S\subseteq S^{r-1}$ be such that for every $v\neq w\in S, \theta_{v,w}\geq \theta_0$. Then $$\#|S|\leq 2\sqrt{3r}\sin{\left(\frac{\theta_0}{2}\right)}^{1-r}\cos{\left(\frac{\theta_0}{2}\right)}^{-1}.$$
\end{lem}
\begin{proof}
Note that balls of radius $\frac{\theta_0}{2}$ (in the spherical distance) about the points of $S$ do not intersect. Thus $$\vol(S^{r-1})\geq \#|S|\cdot \vol\left(B_{\frac{\theta_0}{2}}((1,0,\ldots))\right).$$ But the ball of radius $\frac{\theta_0}{2}$ about $(1,0,\ldots)$ is the spherical cap $x_1\geq \cos{\left(\frac{\theta_0}{2}\right)}$. The surface area of such a cap is $$\frac{1}{2}\vol(S^{r-1})I_{\sin^2{\left(\frac{\theta_0}{2}\right)}}\left(\frac{r-1}{2},\frac{1}{2}\right),$$ where $I_x(a,b)$ is the regularized incomplete beta function.

But $$I_x(a,b) = \frac{x^a(1-x)^b}{a B(a,b)}\left(1 + \sum_{n\geq 0} \frac{B(a+1,n+1)}{B(a+b,b+1)} x^{n+1}\right)$$ where $B(w,z)$ is the usual beta function. Thus in particular $I_x(a,b)\geq \frac{x^a(1-x)^b}{a B(a,b)}$, so that we have found:
\begin{align*}
\#|S|&\leq \frac{(r-1)B\left(\frac{r-1}{2}, \frac{1}{2}\right)}{\sin^{r-1}{\left(\frac{\theta_0}{2}\right)}\cos{\left(\frac{\theta_0}{2}\right)}}
\\&\leq r\sqrt{\pi}\cdot \frac{\Gamma\left(\frac{r-1}{2}\right)}{\Gamma\left(\frac{r}{2}\right)}\cdot \sin^{1-r}{\left(\frac{\theta_0}{2}\right)}\cos{\left(\frac{\theta_0}{2}\right)}^{-1}.
\end{align*}
Therefore it suffices to show that $$\frac{\Gamma\left(\frac{r-1}{2}\right)}{\Gamma\left(\frac{r}{2}\right)}\leq \frac{2\sqrt{3}}{\sqrt{\pi r}}.$$ But this follows via induction (with equality at $r=3$).
\end{proof}

Note that this implies that
\begin{align*}
\max_{S\subseteq \RP^{r-1} : \forall v\neq w\in S, |\langle v, w\rangle|\leq \frac{J}{2}} \#|S|&\leq \sqrt{3r}\sin{\left(\frac{\theta}{2}\right)}^{1-r}\cos{\left(\frac{\theta}{2}\right)}^{-1}
\\&= \sqrt{3r}\left(\frac{1}{2} - \frac{J}{4}\right)^{\frac{1-r}{2}}\left(\frac{1}{2} + \frac{J}{4}\right)^{-\frac{1}{2}},
\end{align*}\noindent
where as usual we have written $J = 2\cos{\theta}$.

Now notice that we have left the $r=2$ case on its own. This is because in this case the unit sphere is simply the circle and we can give a very good estimate for the maximum (the idea is the same):
\begin{lem}
Let $S\subseteq \RP^1$ be such that for every $v\neq w\in S$, $\theta_{v,w}\geq \theta_0$. Then $$\#|S|\leq \floor{\frac{\pi}{\theta_0}}.$$
\end{lem}
\begin{proof}
Let $T := \{\phi\in [0,\pi) : e^{i\phi}\in \pi^{-1}(S)\}$, where $\pi: S^1\to \RP^1$ is the projection. (Note that $\#|T| = \#|S|$.) Without loss of generality $0\in T$. Then the union $$\bigcup_{1\neq t\in T} \left(t - \frac{\theta_0}{2}, t + \frac{\theta_0}{2}\right)\cup \left(\pi - \frac{\theta_0}{2}, \pi\right)\cup \left(0, \frac{\theta_0}{2}\right)$$ is disjoint. On taking measures we find the desired inequality.
\end{proof}

Now for $3\leq r\leq 13$ we use a program written by Henry Cohn to find optimal linear programming bounds on these maxima. This allows us to compile a table of bounds for given $\theta$ ranging from slightly larger than $0$ to slightly smaller than $\frac{\pi}{3}$. Then for each fixed $r$ we choose $c,D,s,J$ making the upper bound on $\#|E(\Z)|$ as small as possible. This choice of $J$ corresponds to a $\theta$ via $J = 2\cos{\theta}$, and one needs only check the sphere packing upper bound we use with rigorous arithmetic for this $J$.\footnote{We end up simply choosing $c = 0.998114, D = 612.117, s = 3$ and instead only optimizing $J$ for each $3\leq r\leq 13$.}

In any case, what is left is simply a Mathematica calculation, and the relevant Mathematica document used to optimize the bound has been included. As a final note, observe that if $(A,B)\equiv (2,2)\pmod{3}$, then $E_{A,B}(\Z) = \emptyset$. Thus we may restrict to the subfamily $\mathcal{G}$ of $(A,B)$ not congruent to $(2,2)$ modulo $3$. Inside this subfamily, we use the methods of Bhargava-Shankar (and Bhargava-Skinner-Zhang) to compute lower bounds for the proportions of curves with rank $0$, $1$, and either $0$ or $1$. For reference, denote by $\tilde{F}_1,\ldots,\tilde{F}_4,\tilde{F}^+,\tilde{F}^-$, and $\tilde{F}$ the subfamilies of curves with $(A,B)\not\equiv (2,2)\pmod{3}$ corresponding to the large families $F_1,\ldots,F_4,F^+,F^-$, and $F$ constructed in \cite{bhargavashankarfive}. Then $\tilde{F}_2,\tilde{F}_3,\tilde{F}_4$ have unchanged densities, and $\tilde{F}_1$ has density $\frac{9}{8}\mu(F_1) - \frac{1}{8}\geq 66.45\%$ (we are lucky because the local root number at $3$ does not vary when $v_3(A) = v_3(B) = 0$). Here we have written $\mu$ to mean the density of a subfamily (where the ambient family is understood). This results in lower bounds of $\mu(\tilde{F}^+)\geq 41.15\%$ and $\mu(\tilde{F}^-)\geq 65.56\%$. Therefore the union of these families has density $\mu(\tilde{F})\geq 60.67\%$. Following Bhargava-Skinner-Zhang, this results in a proportion of at least $22.821\%$ of curves in $\mathcal{G}$ having rank $1$. Following Bhargava-Shankar, this also results in a proportion of at least $22.75\%$ of curves having rank $0$, and at least $84.22\%$ having rank either $0$ or $1$. Since $\mathcal{G}$ has density $\frac{8}{9}$ in $\mathcal{F}_{\mathrm{universal}}$, we in effect gain a factor of $\frac{8}{9}$ (as well as slightly more from the improved lower bounds on rank $\leq 1$ curves) due to these considerations. The remaining optimization is in the Mathematica file.

\bibliography{theaverageellipticcurvehasfewintegralpoints}{}
\bibliographystyle{plain}

\ \\

\end{document}